\documentclass[12pt]{amsart}

\usepackage{amsmath,amsthm,amssymb,amsfonts,mathtools}
\usepackage[margin=2cm]{geometry}
\usepackage{enumerate}
\usepackage{mathrsfs}
\usepackage{tikz}
\usepackage{hyperref}
\hypersetup{hidelinks}

\newcommand{\ZZ}{\mathbb Z}
\newcommand{\NN}{\mathbb N}
\newcommand{\cT}{\mathcal T}
\newcommand{\cA}{\mathcal A}
\newcommand{\cB}{\mathcal B}
\newcommand{\cJ}{\mathcal J}
\newcommand{\cN}{\mathcal N}

\newcommand{\ee}{\mathbf e}
\newcommand{\xx}{\mathbf x}
\newcommand{\uu}{\mathbf u}
\newcommand{\vv}{\mathbf v}
\newcommand{\zz}{\mathbf z}
\newcommand{\etab}{\boldsymbol\eta}
\newcommand{\rhob}{\boldsymbol\rho}
\newcommand{\betab}{\boldsymbol\beta}
\newcommand{\ideal}[1]{\mathfrak i_{#1}}
\newcommand{\Chain}[1]{\mathsf C_{#1}}
\DeclareMathOperator{\rank}{rank}
\DeclareMathOperator{\reg}{reg}

\newtheorem{Theorem}{Theorem}[section]
\newtheorem{Lemma}[Theorem]{Lemma}
\newtheorem{Corollary}[Theorem]{Corollary}
\newtheorem{Proposition}[Theorem]{Proposition}
\theoremstyle{definition}
\newtheorem{Definition}[Theorem]{Definition}
\newtheorem{Setup}[Theorem]{Setup}
\newtheorem{Remark}[Theorem]{Remark}
\newtheorem{Example}[Theorem]{Example}
\newtheorem{Discussion}[Theorem]{Discussion}

\begin{document}
	
	\title[Rank identifications in join-meet ideals]{Rank identifications in join-meet ideals of thin distributive lattices}

\author[T.~Hibi]{Takayuki Hibi}
\address[Takayuki Hibi]
{Department of Pure and Applied Mathematics, 
Graduate School of Information Science and Technology, 
The University of Osaka, 
Suita, Osaka 565-0871, Japan}
\email{hibi@math.sci.osaka-u.ac.jp}

\author[F. Navarra]{Francesco Navarra}      
\address[Francesco Navarra]{Sabanci University, Faculty of Engineering and Natural Sciences, Orta Mahalle, Tuzla 34956, Istanbul, Turkey}	
\email{francesco.navarra@sabanciuniv.edu}

\author[A. A. Qureshi]{Ayesha Asloob Qureshi}      
\address[Ayesha Asloob Qureshi]{Sabanci University, Faculty of Engineering and Natural Sciences, Orta Mahalle, Tuzla 34956, Istanbul, Turkey}	
\email{aqureshi@sabanciuniv.edu, ayesha.asloob@sabanciuniv.edu}

	\keywords{Distributive lattice, join-meet ideal, Hibi ring, prime ideal, lattice ideal, Cohen--Macaulay ring, linear resolution}
	
	\subjclass[2020]{06D05, 13C70, 13H10, 13P10}

    	\thanks{This research was supported through the program “Oberwolfach Research Fellows” (OWRF) by the Mathematisches Forschungsinstitut Oberwolfach in 2026. The second and third authors are supported by Scientific and Technological Research Council of Turkey T\"UB\.{I}TAK under the Grants 126F175 and 124F113, respectively; they are thankful to T\"UB\.{I}TAK for their support. The second author is a member of GNSAGA Indam and he acknowledges their support.} 
        
	\begin{abstract}
Let $L$ be a thin distributive lattice of rank $m+1$, and let $I_L$ be its join-meet ideal in the polynomial ring $S_L$. For $A\subseteq\{1,\dots,m\}$, let $J_A$ be the ideal obtained from $I_L$ by identifying, for every $i\in A$, the variables of the two elements of $L$ of rank $i$.  We prove that $J_A$ is prime if and only if $A$ contains no turning rank and does not contain both ranks adjacent to a turning rank. In this case, $S_L/J_A$ is the toric ring obtained from Hibi's parametrization of the Hibi ring of $L$ by identifying suitable pairs of variables, and for an arbitrary $A$ we describe the saturation of $J_A$ as a lattice ideal. 
	\end{abstract}
	
	\maketitle
	
	\section{Introduction}
	
	Let $L$ be a finite distributive lattice, let $K$ be a field and let $S_L=K[x_\alpha:\alpha\in L]$ be the polynomial ring whose variables are indexed by the elements of $L$. The \emph{join-meet ideal} $I_L$ of $L$ is the ideal of $S_L$ generated by the binomials $x_\alpha x_\beta-x_{\alpha\wedge\beta}x_{\alpha\vee\beta}$, where $\alpha,\beta\in L$ are incomparable. It was introduced by Hibi \cite{Hibi}, who proved that $K[L]=S_L/I_L$ is a normal Cohen--Macaulay domain and an algebra with straightening laws on $L$. Since then, join-meet ideals and their quotient rings, nowadays called Hibi rings, have been studied from many points of view; see, for instance, \cite{HH11,HHO} and the references therein.
	
	In this paper we focus on the thin distributive lattices, also known as generalized snake posets \cite{vonBell}. A finite distributive lattice $L$ of rank $m+1$ is thin if it is simple and has exactly two elements of rank $i$ for every $i\in[m]:=\{1,\dots,m\}$. Denoting these two elements by $c_i$ and $d_i$, we study what happens to the Hibi ring when the two elements of some ranks are identified. More precisely, for $A\subseteq[m]$ we consider the ideal
	\[
	J_A=I_L+(x_{c_i}-x_{d_i}:i\in A)\subseteq S_L,
	\]
	and we ask when $J_A$ is a prime ideal. The answer depends on the \emph{turning ranks} of $L$, namely the integers $2\leq k\leq m-1$ such that $c_{k-1}\vee d_{k-1}=c_{k+1}\wedge d_{k+1}$. Our main result, Theorem~\ref{Thm for primality}, states that $J_A$ is prime if and only if $A$ contains no turning rank and, for every turning rank $k$, the set $A$ does not contain both $k-1$ and $k+1$. Interestingly, although $J_A$ need not be prime, the quotient
$S_L/J_A$ is Cohen--Macaulay for every $A\subseteq[m]$;
see Remark~\ref{rem:homological}. Finally, when $A$ is admissible, we describe $S_L/J_A$ as the toric ring obtained from Hibi's parametrization of $K[L]$ by identifying suitable pairs of variables. This parametrization has a natural probabilistic meaning in terms of the conjunctive Bayesian networks introduced in \cite{BES}, where it is shown that, in suitable coordinates, such a model is the toric variety of the Hibi ring of $L$. In this setting, a rank identification corresponds to requiring that two events have the same conditional probability. Our results then say that the homogeneous relations among the probabilities of the resulting submodel are generated by the join-meet relations together with the obvious linear ones if and only if $A$ is admissible, and that in general they are described by the saturation of $J_A$; see Remark~\ref{rem:CBN} and Example~\ref{ex:CBN}.
    

	The paper is organized as follows. In Section~\ref{sec:basics} we fix the terminology and describe the structure of thin distributive lattices: they are obtained by gluing diamonds and, roughly speaking, they look like staircases whose steps are products of a chain and a chain with two elements. In Section~\ref{sec:necessity} we prove that the two conditions above are necessary, by exhibiting explicit zero-divisors modulo $J_A$, and we show that $J_A$ is prime if and only if it coincides with the lattice ideal of the lattice $\Lambda_A$ spanned by the exponent vectors of its generators. Section~\ref{sec:sufficiency} contains the proof of the sufficiency, which is the core of the paper: we show that every binomial of the lattice ideal of $\Lambda_A$ of degree at least three is generated by binomials of smaller degree, by induction on the rank of $L$ and by means of suitable linear functionals vanishing on $\Lambda_A$, which we call $A$-weights. Finally, Section~\ref{sec:toric} is devoted to the toric parametrization and to its probabilistic interpretation.
	
	\section{Thin distributive lattices and rank identifications}\label{sec:basics}
	
	\subsection{Posets and distributive lattices}
	
	
	A \emph{partially ordered set}, or \emph{poset}, is a set $P$ endowed with a relation $\leq$ which is reflexive, antisymmetric and transitive. Throughout the paper, all posets are assumed to be finite. Two elements $p,q\in P$ are said to be \emph{comparable} if $p\leq q$ or $q\leq p$, and \emph{incomparable} otherwise.

	
	A \emph{chain} of a finite poset $P$ is a totally ordered subset of $P$, and its \emph{length} is its cardinality minus one. The \emph{rank} of $P$, denoted by $\rank(P)$, is the maximal length of a chain of $P$. For $r\geq1$ we denote by $\Chain r$ the chain $\{0<1<\cdots<r-1\}$. If $P$ and $Q$ are posets, $P\times Q$ denotes the Cartesian product of $P$ and $Q$ ordered componentwise, that is, $(p,q)\leq(p',q')$ if and only if $p\leq p'$ in $P$ and $q\leq q'$ in $Q$. A subset $U$ of $P$ is a \emph{poset ideal} if $q\in U$ whenever $q\leq p$ for some $p\in U$, and $\cJ(P)$ denotes the set of poset ideals of $P$, ordered by inclusion.
	
	A \emph{finite lattice} is a finite poset $L$ in which any two elements $\alpha,\beta$ have a greatest lower bound $\alpha\wedge\beta$ and a least upper bound $\alpha\vee\beta$, called the \emph{meet} and the \emph{join} of $\alpha$ and $\beta$, respectively. It is \emph{distributive} if meet and join distribute over each other. We denote by $0_L$ and $1_L$ the least and the greatest element of $L$, and for $\alpha\leq\beta$ we set $[\alpha,\beta]=\{\gamma\in L:\alpha\leq\gamma\leq\beta\}$. If $\alpha<\beta$ and $[\alpha,\beta]=\{\alpha,\beta\}$, we say that $\beta$ \emph{covers} $\alpha$, and we call $\{\alpha,\beta\}$ an \emph{edge} of $L$, since it corresponds to an edge of the Hasse diagram of $L$. A \emph{diamond} is a poset isomorphic to $\Chain2\times\Chain2$, that is, a set $\{a,c,d,b\}$ with $a<c<b$, $a<d<b$ and $c,d$ incomparable. Its four edges are the \emph{lower edges} $\{a,c\}$, $\{a,d\}$ and the \emph{upper edges} $\{c,b\}$, $\{d,b\}$; two of them are \emph{adjacent} if they share an element, and \emph{opposite} otherwise. An element $\alpha\in L$ is an \emph{apex} of $L$ if it is comparable with every element of $L$, and $L$ is \emph{simple} if its only apexes are $0_L$ and $1_L$. An element $0_L\neq p\in L$ is \emph{join-irreducible} if $p=\alpha\vee\beta$ implies $p=\alpha$ or $p=\beta$. We denote by $P_L$ the subposet of $L$ consisting of its join-irreducible elements and, for $\alpha\in L$, we set
	\[
	\ideal\alpha=\{p\in P_L:p\leq\alpha\}.
	\]
	By Birkhoff's fundamental structure theorem \cite[Theorem~9.1.7]{HH11}, if $L$ is a finite distributive lattice,  the map $\alpha\mapsto\ideal\alpha$ is an isomorphism of lattices from $L$ onto $\cJ(P_L)$; in particular, $\ideal{\alpha\wedge\beta}=\ideal\alpha\cap\ideal\beta$ and $\ideal{\alpha\vee\beta}=\ideal\alpha\cup\ideal\beta$ for all $\alpha,\beta\in L$. The \emph{rank} of $\alpha\in L$, denoted by $\rank_L(\alpha)$, is the maximal length of a chain of $L$ whose greatest element is $\alpha$, and $\rank(L)=\rank_L(1_L)$. We shall freely use the following elementary facts (see, e.g., \cite[Chapter~3]{Stanley}).
	
	\begin{Remark}\label{lem:basic}\rm
		Let $L$ be a finite distributive lattice. Then the following hold.
		\begin{enumerate}[(a)]
			\item $\rank_L(\alpha)=|\ideal\alpha|$ for every $\alpha\in L$. If $\alpha<\beta$, then $\rank_L(\alpha)<\rank_L(\beta)$, and all the maximal chains of $[\alpha,\beta]$ have length $\rank_L(\beta)-\rank_L(\alpha)$. In particular, two distinct elements of $L$ of the same rank are incomparable, and, for every $\alpha\leq\beta$, the interval $[\alpha,\beta]$ contains an element of rank $\rho$ for every $\rank_L(\alpha)\leq\rho\leq\rank_L(\beta)$.
			\item $\rank_L(\alpha\wedge\beta)+\rank_L(\alpha\vee\beta)=\rank_L(\alpha)+\rank_L(\beta)$ for all $\alpha,\beta\in L$.
			\item For all $\alpha\leq\beta$, the interval $[\alpha,\beta]$ is closed under meet and join, so it is a distributive lattice with the operations of $L$.
		\end{enumerate}
	\end{Remark}

	\subsection{Thin distributive lattices}
	
	For a finite distributive lattice $L$ and an integer $i\geq0$, let $\rho_L(i)$ be the number of elements of $L$ of rank $i$. A finite simple distributive lattice $L$ of rank $d$ is called \emph{thin} if $\rho_L(i)=2$ for every $1\leq i\leq d-1$; see \cite{HH23}.
	Thin distributive lattices are also called \emph{generalized snake posets} \cite{vonBell}. Their structure is easily read off from the definition. We describe it in the following discussion, which also fixes the notation used throughout the paper.
	
	\begin{Discussion}\label{lem:structure}\rm
		From now on, unless otherwise stated, $L$ denotes a thin distributive lattice of rank $m+1$, with $m\geq1$. Thus $L=\{0_L,1_L\}\sqcup\{c_i,d_i:i\in[m]\}$, where $c_i$ and $d_i$ denote the two elements of $L$ of rank $i$, labelled in an arbitrary way; in particular, $|L|=2m+2$. For every $i\in[m]$ we set $a_i=c_i\wedge d_i$ and $b_i=c_i\vee d_i$, and we call $D_i=[a_i,b_i]$ the \emph{elementary diamond} of rank $i$.
		
		The elementary diamonds are the building blocks of $L$. Indeed, for every $i\in[m]$ the interval $D_i$ is a diamond, namely $D_i=\{a_i,c_i,d_i,b_i\}$, with $\rank_L(a_i)=i-1$ and $\rank_L(b_i)=i+1$. Consequently, $a_1=0_L$, $b_m=1_L$ and $L=D_1\cup\cdots\cup D_m$. Moreover, for $1\leq i\leq m-1$ we have $a_{i+1}\in\{c_i,d_i\}$ and $b_i\in\{c_{i+1},d_{i+1}\}$, so that $D_i\cap D_{i+1}=\{a_{i+1},b_i\}$ is an upper edge of $D_i$ and a lower edge of $D_{i+1}$. In other words, $L$ is obtained by gluing $D_1,\ldots,D_m$, each one along an edge of the previous one; see Figure~\ref{fig:thin-lattice}.
		
		Let $2\leq k\leq m-1$. The diamond $D_k$ shares the lower edge $\{a_k,b_{k-1}\}$ with $D_{k-1}$ and the upper edge $\{a_{k+1},b_k\}$ with $D_{k+1}$, where $b_{k-1},a_{k+1}\in\{c_k,d_k\}$. Hence exactly one of the following holds: either $b_{k-1}\neq a_{k+1}$, in which case these two edges of $D_k$ are opposite and $D_{k-1}\cap D_{k+1}=\emptyset$, or $b_{k-1}=a_{k+1}$, in which case they are adjacent and $D_{k-1}\cap D_{k+1}=\{b_{k-1}\}$. In the second case, $L$ changes direction at rank $k$. Accordingly, an integer $k\in\{2,\ldots,m-1\}$ is called a \emph{turning rank} of $L$ if $b_{k-1}=a_{k+1}$, that is, if $c_{k-1}\vee d_{k-1}=c_{k+1}\wedge d_{k+1}$. We denote by $\cT(L)$ the set of turning ranks of $L$; for instance, $\cT(L)=\{3,5,6,7,9,10,11\}$ for the lattice of Figure~\ref{fig:thin-lattice}.
		
		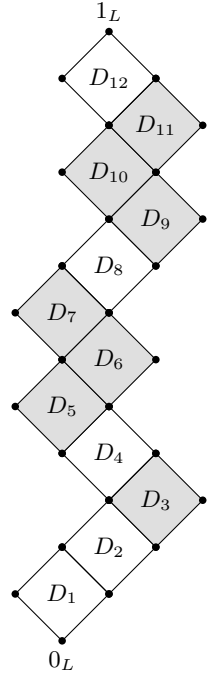
\begin{figure}[htbp]
			\centering
			\begin{tikzpicture}[x=0.62cm,y=0.62cm,dot/.style={circle,fill,inner sep=1pt},every node/.style={font=\scriptsize}]
				\begin{scope}[cm={1,1,-1,1,(0,0)}]
					\foreach \x/\y/\i in {0/0/1,1/0/2,2/0/3,2/1/4,2/2/5,3/2/6,3/3/7,4/3/8,5/3/9,5/4/10,6/4/11,6/5/12}{
						\ifnum\i=3 \fill[gray!25] (\x,\y) rectangle ++(1,1);\fi
						\ifnum\i=5 \fill[gray!25] (\x,\y) rectangle ++(1,1);\fi
						\ifnum\i=6 \fill[gray!25] (\x,\y) rectangle ++(1,1);\fi
						\ifnum\i=7 \fill[gray!25] (\x,\y) rectangle ++(1,1);\fi
						\ifnum\i=9 \fill[gray!25] (\x,\y) rectangle ++(1,1);\fi
						\ifnum\i=10 \fill[gray!25] (\x,\y) rectangle ++(1,1);\fi
						\ifnum\i=11 \fill[gray!25] (\x,\y) rectangle ++(1,1);\fi
						\draw (\x,\y) rectangle ++(1,1);
						\foreach \dx/\dy in {0/0,1/0,0/1,1/1} \node[dot] at (\x+\dx,\y+\dy) {};
						\node at (\x+0.5,\y+0.5) {$D_{\i}$};
					}
					\node[below] at (0,0) {$0_L$};
					\node[above] at (7,6) {$1_L$};
				\end{scope}
			\end{tikzpicture}
			\caption{The Hasse diagram of a thin distributive lattice of rank $13$. The elementary diamonds of turning rank are shaded, so that $\cT(L)=\{3,5,6,7,9,10,11\}$.}
			\label{fig:thin-lattice}
		\end{figure}
		
		An interval of $L$ isomorphic to $\Chain{N+1}\times\Chain2$, for some $N\geq1$, is called a \emph{ladder}. Consecutive elementary diamonds with no turning rank in between are glued along opposite edges, and so they form a ladder. We fix the following notation for ladders, which will be used in several proofs. 
		
		Let $1\leq s\leq t\leq m$ be such that no turning rank $k$ of $L$ satisfies $s<k<t$, and set $N=t-s+1$. Then $[a_s,b_t]=D_s\cup\cdots\cup D_t$ is a ladder, and we denote by $\lambda_j$ and $\mu_j$ the elements of $[a_s,b_t]$ corresponding to $(j,0)$ and $(j,1)$ of $\Chain{N+1}\times\Chain2$, respectively, for $0\leq j\leq N$; see Figure~\ref{fig:ladder}. Thus $\lambda_j=a_{s+j}$ for $0\leq j\leq N-1$ and $\mu_j=b_{s+j-1}$ for $1\leq j\leq N$, and for every $j\in[N]$ we have $D_{s+j-1}=[\lambda_{j-1},\mu_j]=\{\lambda_{j-1},\mu_{j-1},\lambda_j,\mu_j\}$ and $\{c_{s+j-1},d_{s+j-1}\}=\{\mu_{j-1},\lambda_j\}$. Moreover, computing componentwise in $\Chain{N+1}\times\Chain2$, we see that $\lambda_i\leq\mu_j$ if and only if $i\leq j$, that $\mu_i\nleq\lambda_j$ for all $0\leq i,j\leq N$, and that, for $0\leq i<j\leq N$, the elements $\mu_i$ and $\lambda_j$ are incomparable, with $\mu_i\wedge\lambda_j=\lambda_i$ and $\mu_i\vee\lambda_j=\mu_j$.
		
		\begin{figure}[htbp]
			\centering
			\begin{tikzpicture}[x=0.55cm,y=0.55cm,dot/.style={circle,fill,inner sep=1.2pt},every node/.style={font=\footnotesize}]
				\begin{scope}[cm={1,1,-1,1,(0,0)}]
					\foreach \k in {1,...,5} \draw (\k-1,0) rectangle (\k,1);
					\foreach \k in {0,...,5} {\node[dot] at (\k,0) {}; \node[dot] at (\k,1) {};}
					\foreach \k in {0,...,5} \node[below right] at (\k,0) {$\lambda_{\k}$};
					\foreach \k in {0,...,5} \node[above left] at (\k,1) {$\mu_{\k}$};
				\end{scope}
			\end{tikzpicture}
			\caption{The Hasse diagram of a ladder $\Chain{6}\times\Chain2$.}
			\label{fig:ladder}
		\end{figure}
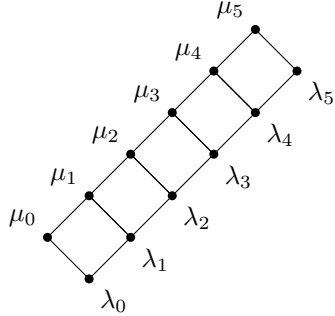
		
		Finally, let $k_1<\cdots<k_q$ be the turning ranks of $L$, and set $k_0=1$ and $k_{q+1}=m$. Then the intervals $[a_{k_j},b_{k_{j+1}}]$, $0\leq j\leq q$, are ladders, two consecutive ones share exactly the elementary diamond $D_{k_j}$, and $L$ is their union. Roughly speaking, $L$ is a staircase whose steps are ladders, glued along the elementary diamonds of turning rank; see Figure~\ref{fig:thin-lattice}. Since, for $k\in\cT(L)$, the element $b_{k-1}=a_{k+1}$ is comparable with every element of $L$, two incomparable elements of $L$ always lie in a common ladder, where they are of the form $\mu_i$ and $\lambda_j$ with $i<j$.
	\end{Discussion}

	\subsection{Join-meet ideals and rank identifications}
	
	Let $K$ be a field, and let $S_L=K[x_\alpha:\alpha\in L]$ be the polynomial ring over $K$ whose variables are indexed by the elements of $L$, with the standard grading $\deg x_\alpha=1$. For us, $\NN=\{0,1,2,\ldots\}$. For $\uu\in\NN^L$ we denote by $\uu(w)$ the coordinate of $\uu$ indexed by $w\in L$, and we set $\xx^\uu=\prod_{w\in L}x_w^{\uu(w)}$, so that every monomial of $S_L$ is of the form $\xx^\uu$ for a unique $\uu\in\NN^L$. By a \emph{binomial} we mean a polynomial $\xx^\uu-\xx^\vv$ with $\uu,\vv\in\NN^L$ and $\uu\neq\vv$. The \emph{join-meet ideal} of $L$ is
	\[
	I_L=(x_\alpha x_\beta-x_{\alpha\wedge\beta}x_{\alpha\vee\beta}:\alpha,\beta\in L\text{ incomparable})\subseteq S_L,
	\]
	and its generators are called the \emph{join-meet binomials} of $L$. By \cite{Hibi}, the \emph{Hibi ring} $K[L]=S_L/I_L$ is a normal Cohen--Macaulay domain of dimension $|P_L|+1$, and the standard monomials $x_{\alpha_1}x_{\alpha_2}\cdots x_{\alpha_q}$, with $\alpha_1\leq\alpha_2\leq\cdots\leq\alpha_q$, form a $K$-basis of $K[L]$. Since $|P_L|=\rank_L(1_L)=m+1$ by Remark~\ref{lem:basic}(a), we have $\dim K[L]=m+2$.
	
	For every $i\in[m]$, we consider the linear binomial $\ell_i=x_{c_i}-x_{d_i}$ and, for $A\subseteq[m]$, we set
	\[
	J_A=I_L+(\ell_i:i\in A).
	\]
	Thus $S_L/J_A$ is obtained from $K[L]$ by identifying the two elements of rank $i$ for every $i\in A$. Note that $J_A$ does not depend on the labelling of the elements of rank $i$, since exchanging $c_i$ and $d_i$ replaces $\ell_i$ by $-\ell_i$.
	
	\begin{Definition}\label{Def admissible}
		A subset $A\subseteq[m]$ is \emph{admissible} for $L$ if the following conditions hold:
		\begin{enumerate}[(i)]
			\item $A\cap\cT(L)=\emptyset$;
			\item $\{k-1,k+1\}\nsubseteq A$ for every $k\in\cT(L)$.
		\end{enumerate}
	\end{Definition}
	
	Our main result is the following.
	
	\begin{Theorem}\label{Thm for primality}
		Let $L$ be a thin distributive lattice of rank $m+1$, and let $A\subseteq[m]$. Then $J_A$ is a prime ideal if and only if $A$ is admissible.
	\end{Theorem}

	\section{Necessary conditions and the associated lattice ideal}\label{sec:necessity}
	
	In this section we prove that the conditions of Definition~\ref{Def admissible} are necessary for the primality of $J_A$, and we show that $J_A$ is prime if and only if it is a lattice ideal. We first fix the notation for the elements of $L$ around a turning rank, which will be used throughout the section; all the relations below can be read off from Figure~\ref{fig:turning-rank}.
	
	\begin{Setup}\label{lem:turning-config}\rm
		Let $k\in\cT(L)$. Set $d=b_{k-1}=a_{k+1}$, $c=a_k$, $f=b_k$, $a=a_{k-1}$ and $v=b_{k+1}$, and let $e$, $b$ and $u$ be the elements of rank $k$, $k-1$ and $k+1$ different from $d$, $c$ and $f$, respectively. Then $D_{k-1}=\{a,b,c,d\}$, $D_k=\{c,d,e,f\}$ and $D_{k+1}=\{d,u,f,v\}$.
		Moreover, each of the pairs $\{b,c\}$, $\{d,e\}$, $\{u,f\}$, $\{b,e\}$, $\{u,e\}$ consists of incomparable elements, and we have the following relations: $b\wedge c=a$ and $b\vee c=d$, $d\wedge e=c$ and $d\vee e=f$, $u\wedge f=d$ and $u\vee f=v$, $b\wedge e=a$ and $b\vee e=f$, $u\wedge e=c$ and $u\vee e=v$. 	Finally, $\ell_{k-1}=\pm(x_b-x_c)$, $\ell_k=\pm(x_d-x_e)$ and $\ell_{k+1}=\pm(x_u-x_f)$. Note that $a=0_L$ if $k=2$, and $v=1_L$ if $k=m-1$.
	\end{Setup}
	
	\begin{figure}[htbp]
		\centering
		\begin{tikzpicture}[x=1.35cm,y=1cm,dot/.style={circle,fill,inner sep=1.3pt},every node/.style={font=\small}]
			\coordinate (a) at (-1,0);
			\coordinate (b) at (-2,1); \coordinate (c) at (0,1);
			\coordinate (d) at (-1,2); \coordinate (e) at (1,2);
			\coordinate (u) at (-2,3); \coordinate (f) at (0,3);
			\coordinate (v) at (-1,4);
			\foreach \s/\t in {a/b,a/c,b/d,c/d,c/e,d/u,d/f,e/f,u/v,f/v} \draw (\s)--(\t);
			\foreach \s in {a,b,c,d,e,f,u,v} \node[dot] at (\s) {};
			\node[below] at (a) {$a$}; \node[left] at (b) {$b$};
			\node[below right] at (c) {$c$}; \node[left=3pt] at (d) {$d$};
			\node[right] at (e) {$e$}; \node[left] at (u) {$u$};
			\node[above right] at (f) {$f$}; \node[above] at (v) {$v$};
			\node at (-1,1) {$D_{k-1}$}; \node at (0,2) {$D_{k}$}; \node at (-1,3) {$D_{k+1}$};
			\foreach \r/\y in {k-2/0,k-1/1,k/2,k+1/3,k+2/4} \node[right] at (2.2,\y) {rank $\r$};
		\end{tikzpicture}
		\caption{The elementary diamonds $D_{k-1}$, $D_k$ and $D_{k+1}$ around a turning rank $k$.}
		\label{fig:turning-rank}
	\end{figure}
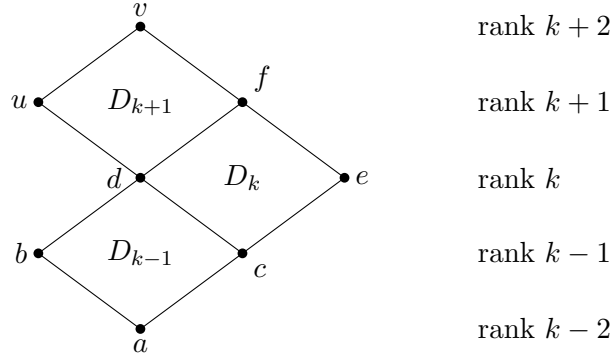
	
	The next lemma provides a tool to show that certain binomials do not belong to $J_A$.
	
	\begin{Lemma}\label{lem:separation}
		Let $A\subseteq[m]$, let $k\in\cT(L)$, and use the notation of Setup~\ref{lem:turning-config}. Set $E_a=\{c_{k-2},d_{k-2}\}$ if $k-2\in A$, and $E_a=\{a\}$ otherwise; similarly, set $E_v=\{c_{k+2},d_{k+2}\}$ if $k+2\in A$, and $E_v=\{v\}$ otherwise. Let $\cN=K[\varepsilon,\eta]/(\varepsilon^2,\eta^2)$, and let $\chi:S_L\to\cN$ be the homomorphism of $K$-algebras defined by
		\[
		\chi(x_w)=
		\begin{cases}
			\varepsilon,&w\in E_a,\\
			\eta,&w\in E_v,\\
			0,&w\notin E_a\cup E_v.
		\end{cases}
		\]
		Then $J_A\subseteq\ker\chi$, while $\chi(x_ax_v)=\varepsilon\eta\neq0$ and $\chi(x_bx_u)=\chi(x_e^2)=0$.
	\end{Lemma}
	
	\begin{proof}
		All the elements of $E_a$ have rank $k-2$ and all the elements of $E_v$ have rank $k+2$; in particular, $E_a\cap E_v=\emptyset$ and $\chi$ is well defined. Moreover, $a\in E_a$ and $v\in E_v$, since $a$ has rank $k-2$, $v$ has rank $k+2$, and $\{c_{k-2},d_{k-2}\}$, respectively $\{c_{k+2},d_{k+2}\}$, is the set of all the elements of that rank whenever $k-2\in A$, respectively $k+2\in A$. Since $b,u,e$ have rank $k-1$, $k+1$ and $k$, we get $\chi(x_ax_v)=\varepsilon\eta$ and $\chi(x_bx_u)=\chi(x_e^2)=0$.
		
		Let $i\in A$. If $i\notin\{k-2,k+2\}$, then $c_i,d_i\notin E_a\cup E_v$, and so $\chi(\ell_i)=0$. If $i=k-2$, then $c_i,d_i\in E_a$ and $\chi(\ell_i)=0$, and similarly if $i=k+2$. Therefore $\ell_i\in\ker\chi$ for every $i\in A$.
		Observe that $d=b_{k-1}=a_{k+1}$ is greater than every element of rank smaller than $k$ and smaller than every element of rank greater than $k$, as one sees from Figure~\ref{fig:thin-lattice}. In particular, every element of $E_a$ is smaller than every element of $E_v$.
		
		Now note that, for $w,w'\in L$, we have $\chi(x_wx_{w'})\neq0$ if and only if one of $w,w'$ belongs to $E_a$ and the other one belongs to $E_v$, since $\varepsilon^2=\eta^2=0$. Let $\alpha,\beta\in L$ be incomparable. By the previous paragraph, $\alpha$ and $\beta$ cannot belong one to $E_a$ and the other one to $E_v$, so $\chi(x_\alpha x_\beta)=0$. Assume that $\chi(x_{\alpha\wedge\beta}x_{\alpha\vee\beta})\neq0$. Since $\rank_L(\alpha\wedge\beta)<\rank_L(\alpha\vee\beta)$, we must have $\alpha\wedge\beta\in E_a$ and $\alpha\vee\beta\in E_v$, so that $\rank_L(\alpha\wedge\beta)=k-2$ and $\rank_L(\alpha\vee\beta)=k+2$. Since $\alpha$ and $\beta$ are incomparable, their ranks lie strictly between $k-2$ and $k+2$, and $\rank_L(\alpha)+\rank_L(\beta)=2k$ by Remark~\ref{lem:basic}(b). Hence either $\alpha$ and $\beta$ both have rank $k$, or one of them has rank $k-1$ and the other one has rank $k+1$. In the first case $\{\alpha,\beta\}=\{d,e\}$, and so $\alpha\wedge\beta=c$ has rank $k-1$, a contradiction. In the second case, with the notation of Setup~\ref{lem:turning-config}, one of $\alpha,\beta$ belongs to $\{b,c\}$ and the other one to $\{u,f\}$; since $b,c<d<u,f$, the elements $\alpha$ and $\beta$ are comparable, again a contradiction. Hence $\chi(x_\alpha x_\beta-x_{\alpha\wedge\beta}x_{\alpha\vee\beta})=0$ for all incomparable $\alpha,\beta$, and $J_A\subseteq\ker\chi$.
	\end{proof}
	
	We are ready to prove that the conditions of Definition~\ref{Def admissible} are necessary for the primality of $J_A$.
	
	\begin{Proposition}\label{prop:turning-relation-nonprime}
		Let $L$ be a thin distributive lattice of rank $m+1$, and let $A\subseteq[m]$. If $A\cap\cT(L)\neq\emptyset$, then $J_A$ is not prime.
	\end{Proposition}
	
	\begin{proof}
		Choose $k\in A\cap\cT(L)$, use the notation of Setup~\ref{lem:turning-config}, and consider the binomial $g=x_ax_v-x_bx_u$. We prove that $x_cg\in J_A$, while neither $x_c$ nor $g$ belongs to $J_A$. By Setup~\ref{lem:turning-config}, the binomials $x_ux_e-x_cx_v$ and $x_bx_c-x_ax_d$ are join-meet binomials of $L$, and $x_d-x_e=\pm\ell_k\in J_A$. Since
		\[
		x_cg=x_a(x_cx_v-x_ux_e)+x_u(x_ax_d-x_bx_c)-x_ax_u(x_d-x_e),
		\]
		we get $x_cg\in J_A$.
		
		Let $\epsilon:S_L\to K$ be the homomorphism of $K$-algebras such that $\epsilon(x_w)=1$ for every $w\in L$. Every generator of $J_A$ is a difference of two monomials, so $J_A\subseteq\ker\epsilon$, while $\epsilon(M)=1$ for every monomial $M$. Hence $J_A$ contains no monomial and, in particular, $x_c\notin J_A$. Finally, with the notation of Lemma~\ref{lem:separation}, we have $J_A\subseteq\ker\chi$ and $\chi(g)=\chi(x_ax_v)-\chi(x_bx_u)=\varepsilon\eta\neq0$, so $g\notin J_A$. Therefore $J_A$ is not prime.
	\end{proof}
	
	\begin{Proposition}\label{prop:neighboring-relations-nonprime}
		Let $L$ be a thin distributive lattice of rank $m+1$, and let $A\subseteq[m]$. If there exists $k\in\cT(L)$ such that $\{k-1,k+1\}\subseteq A$, then $J_A$ is not prime.
	\end{Proposition}
	
	\begin{proof}
		Choose $k\in\cT(L)$ with $k-1,k+1\in A$, use the notation of Setup~\ref{lem:turning-config}, and consider the binomial $g=x_ax_v-x_e^2$. By Setup~\ref{lem:turning-config}, the binomials $x_ux_f-x_dx_v$, $x_bx_e-x_ax_f$ and $x_dx_e-x_cx_f$ are join-meet binomials of $L$, and $x_b-x_c=\pm\ell_{k-1}$ and $x_u-x_f=\pm\ell_{k+1}$ belong to $J_A$. Since
		\[
		x_dg=x_a(x_dx_v-x_ux_f)+x_ax_f(x_u-x_f)+x_f(x_ax_f-x_bx_e)+x_fx_e(x_b-x_c)+x_e(x_cx_f-x_dx_e),
		\]
		we get $x_dg\in J_A$. As in the proof of Proposition~\ref{prop:turning-relation-nonprime}, $J_A$ contains no monomial, so $x_d\notin J_A$. Moreover, with the notation of Lemma~\ref{lem:separation}, we have $\chi(g)=\chi(x_ax_v)-\chi(x_e^2)=\varepsilon\eta\neq0$, so $g\notin J_A$. Therefore $J_A$ is not prime.
	\end{proof}
	
To prove the converse, we associate with $J_A$ a sublattice
of $\mathbb{Z}^L$. To avoid any confusion with the distributive lattice $L$, we denote it by $\Lambda_A$. Let $\{\ee_w:w\in L\}$ be the canonical basis of $\ZZ^L$. 
	 For $i\in[m]$ we set
	\[
	\etab_i=\ee_{a_i}+\ee_{b_i}-\ee_{c_i}-\ee_{d_i},\qquad\rhob_i=\ee_{c_i}-\ee_{d_i},
	\]
	and, for $A\subseteq[m]$, we let $\Lambda_A$ be the subgroup of $\ZZ^L$ generated by $\{\etab_i:i\in[m]\}\cup\{\rhob_i:i\in A\}$.
	
	\begin{Lemma}\label{lem:integral-basis-LA}
		For every $A\subseteq[m]$, there exists a subset $F_A\subseteq L$ such that
		\[
		\cB_A=\{\etab_i:i\in[m]\}\cup\{\rhob_i:i\in A\}\cup\{\ee_w:w\in F_A\}
		\]
		is a $\ZZ$-basis of $\ZZ^L$. In particular, $\Lambda_A$ is saturated, $\{\etab_i:i\in[m]\}\cup\{\rhob_i:i\in A\}$ is a $\ZZ$-basis of $\Lambda_A$, and $\rank_\ZZ\Lambda_A=m+|A|$.
	\end{Lemma}
	
	\begin{proof}
		For $2\leq i\leq m$ we set $\Delta_i=D_i\setminus(D_1\cup\cdots\cup D_{i-1})$. The elements of $D_j$ have rank at most $j+1$, so the elements of $D_1\cup\cdots\cup D_{i-1}$ have rank at most $i$, and the only one of rank $i$ is $b_{i-1}$. Since $D_i=\{a_i,c_i,d_i,b_i\}$, with $a_i\in\{c_{i-1},d_{i-1}\}\subseteq D_{i-1}$, $b_{i-1}\in\{c_i,d_i\}$ and $\rank_L(b_i)=i+1$, we get $\Delta_i=\{p_i,q_i\}$, where $p_i=b_i$ and $q_i$ is the element of $\{c_i,d_i\}$ different from $b_{i-1}$. We also set $p_1=a_1$ and $q_1=c_1$. Since every element of $L$ belongs to some elementary diamond, we have the decomposition
		$
		L=D_1\sqcup\Delta_2\sqcup\cdots\sqcup\Delta_m.
		$
		Define
		\[
		P_A=\{p_i:i\in[m]\}\cup\{q_i:i\in A\},\qquad F_A=\{b_1,d_1\}\cup\{q_i:i\in[m]\setminus A\}.
		\]
		Since $D_1=\{p_1,q_1,b_1,d_1\}$, the above decomposition gives $L=P_A\sqcup F_A$, with $|P_A|=m+|A|$ and $|F_A|=m+2-|A|$.
		
		For every $i\in[m]$, the element $p_i$ is the least or the greatest element of $D_i$, while $q_i\in\{c_i,d_i\}$. Hence $\etab_i(p_i)=1$, $\etab_i(q_i)=-1$, $\rhob_i(p_i)=0$ and $\rhob_i(q_i)=\kappa_i$, where $\kappa_i=1$ if $q_i=c_i$ and $\kappa_i=-1$ if $q_i=d_i$. Moreover, if $j>i$, then $p_j,q_j\in\Delta_j$ do not belong to $D_i$, and so $\etab_i(p_j)=\etab_i(q_j)=\rhob_i(p_j)=\rhob_i(q_j)=0$.
		
		Order the generators of $\Lambda_A$ by listing $\etab_i$ and then, if $i\in A$, $\rhob_i$, for $i=1,\ldots,m$, and order the elements of $P_A$ in the same way, by listing $p_i$ and then, if $i\in A$, $q_i$, for $i=1,\ldots,m$. Let $M_A$ be the $(m+|A|)\times(m+|A|)$ matrix whose columns are the generators of $\Lambda_A$ restricted to the coordinates indexed by $P_A$, with respect to these orders. By the previous paragraph, $M_A$ is block upper triangular, and its diagonal block corresponding to the index $i$ is the $1\times1$ matrix $(\etab_i(p_i))=(1)$ if $i\notin A$, and the matrix
		\[
		\begin{pmatrix}
			\etab_i(p_i)&\rhob_i(p_i)\\
			\etab_i(q_i)&\rhob_i(q_i)
		\end{pmatrix}
		=
		\begin{pmatrix}
			1&0\\
			-1&\kappa_i
		\end{pmatrix}
		\]
		if $i\in A$. Hence $\det M_A=\prod_{i\in A}\kappa_i\in\{1,-1\}$.
		
		Now fix an ordering $F_A=\{w_1,\ldots,w_{m+2-|A|}\}$, and consider the square matrix whose columns are the elements of $\cB_A$, listing first the generators of $\Lambda_A$ as above and then $\ee_{w_1},\ldots,\ee_{w_{m+2-|A|}}$, and whose rows are indexed by the elements of $P_A$, as above, followed by $w_1,\ldots,w_{m+2-|A|}$. This matrix has the form
		\[
		\begin{pmatrix}
			M_A&0\\
			N_A&I
		\end{pmatrix},
		\]
		where $N_A$ is an integer matrix and $I$ is the identity matrix, so its determinant is $\det M_A\in\{1,-1\}$. Therefore $\cB_A$ is a $\ZZ$-basis of $\ZZ^L$, and $\ZZ^L=\Lambda_A\oplus\bigoplus_{w\in F_A}\ZZ\ee_w$. Thus $\Lambda_A$ is a direct summand of $\ZZ^L$, hence it is saturated, and the remaining assertions follow.
	\end{proof}
	
	For a subgroup $\Lambda\subseteq\ZZ^L$, the \emph{lattice ideal} of $\Lambda$ is
	\[
	I_\Lambda=(\xx^{\zz^+}-\xx^{\zz^-}:\zz\in\Lambda)\subseteq S_L,
	\]
	where $\zz^+,\zz^-\in\NN^L$ are the positive and the negative part of $\zz$, so that $\zz=\zz^+-\zz^-$. Recall that, for $\uu,\vv\in\NN^L$, the binomial $\xx^\uu-\xx^\vv$ belongs to $I_\Lambda$ if and only if $\uu-\vv\in\Lambda$, and that $I_\Lambda$ is a prime ideal whenever $\Lambda$ is saturated; see, for instance, \cite{HHO}. Since all the generators of $\Lambda_A$ have coordinate sum zero, $I_{\Lambda_A}$ is a homogeneous ideal and the two monomials of every binomial of $I_{\Lambda_A}$ have the same degree.
	
	For an ideal $Q\subseteq S_L$ and $f\in S_L$, we set $Q:f^\infty=\bigcup_{n\geq0}(Q:f^n)$, where $Q:f^n=\{g\in S_L:f^ng\in Q\}$. Moreover, we set $X_L=\prod_{w\in L}x_w$.
	
	\begin{Proposition}\label{prop:saturation-JA}
		Let $L$ be a thin distributive lattice of rank $m+1$, and let $A\subseteq[m]$. Then $J_A\subseteq I_{\Lambda_A}$ and
		\[
		I_{\Lambda_A}=J_A:X_L^\infty.
		\]
		In particular, there exists an integer $N\geq0$ such that $I_{\Lambda_A}=J_A:X_L^N$.
	\end{Proposition}
	
	\begin{proof}
		Consider the ideal
		\[
		I_A^{\mathrm{bas}}=(x_{a_i}x_{b_i}-x_{c_i}x_{d_i}:i\in[m])+(\ell_i:i\in A).
		\]
		The exponent vectors of its generators are, up to sign, the elements $\etab_i$, $i\in[m]$, and $\rhob_i$, $i\in A$, which form a $\ZZ$-basis of the saturated lattice $\Lambda_A$ by Lemma~\ref{lem:integral-basis-LA}. Thus $I_A^{\mathrm{bas}}$ is a lattice basis ideal of $\Lambda_A$, and \cite[Proposition~1.1]{HostenShapiro} gives $I_{\Lambda_A}=I_A^{\mathrm{bas}}:X_L^\infty$.
		
		Since $x_{c_i}x_{d_i}-x_{a_i}x_{b_i}$ is a join-meet binomial of $L$ for every $i\in[m]$, we have $I_A^{\mathrm{bas}}\subseteq J_A$. We also have $J_A\subseteq I_{\Lambda_A}$. Indeed, the exponent vector of $\ell_i$ is $\rhob_i\in\Lambda_A$. Let $\alpha,\beta\in L$ be incomparable. By the last paragraph of Discussion~\ref{lem:structure}, $\alpha$ and $\beta$ lie in a ladder $[a_s,b_t]$ and, after possibly exchanging them, we have $\alpha=\mu_i$ and $\beta=\lambda_j$ for some $0\leq i<j\leq N$, so that $\alpha\wedge\beta=\lambda_i$ and $\alpha\vee\beta=\mu_j$; see Figure~\ref{fig:ladder}. In other words, the join-meet binomial $x_\alpha x_\beta-x_{\alpha\wedge\beta}x_{\alpha\vee\beta}$ corresponds to the rectangle formed by the consecutive squares $D_{s+i},\ldots,D_{s+j-1}$ of the ladder, and we show that its exponent vector is the sum of the vectors $\etab$ of these squares. Indeed, for every $y\in[N]$ the elementary diamond $D_{s+y-1}=\{\lambda_{y-1},\mu_{y-1},\lambda_y,\mu_y\}$ has least element $\lambda_{y-1}$ and greatest element $\mu_y$, so that $\etab_{s+y-1}=\ee_{\lambda_{y-1}}+\ee_{\mu_y}-\ee_{\mu_{y-1}}-\ee_{\lambda_y}$. Summing over $i<y\leq j$, the intermediate terms cancel in pairs, and we obtain
		\[
		\ee_{\alpha\wedge\beta}+\ee_{\alpha\vee\beta}-\ee_\alpha-\ee_\beta=\ee_{\lambda_i}+\ee_{\mu_j}-\ee_{\mu_i}-\ee_{\lambda_j}=\sum_{y=i+1}^{j}\etab_{s+y-1}\in\Lambda_A.
		\]
		
		Saturating with respect to $X_L$, we obtain
		\[
		I_{\Lambda_A}=I_A^{\mathrm{bas}}:X_L^\infty\subseteq J_A:X_L^\infty\subseteq I_{\Lambda_A}:X_L^\infty=I_{\Lambda_A},
		\]
		where the last equality holds because $I_{\Lambda_A}:X_L^\infty=(I_A^{\mathrm{bas}}:X_L^\infty):X_L^\infty=I_A^{\mathrm{bas}}:X_L^\infty$. Hence $I_{\Lambda_A}=J_A:X_L^\infty$. Finally, since $S_L$ is Noetherian, the ascending chain $J_A:X_L\subseteq J_A:X_L^2\subseteq\cdots$ stabilizes, and so $J_A:X_L^\infty=J_A:X_L^N$ for some $N\geq0$.
	\end{proof}
	
	\begin{Theorem}\label{thm:prime-iff-lattice-ideal}
		Let $L$ be a thin distributive lattice of rank $m+1$, and let $A\subseteq[m]$. Then the following conditions are equivalent:
		\begin{enumerate}[(i)]
			\item $J_A$ is prime;
			\item $J_A=I_{\Lambda_A}$;
			\item $J_A:X_L^\infty=J_A$.
		\end{enumerate}
	\end{Theorem}
	
	\begin{proof}
		(ii)$\Leftrightarrow$(iii) This follows from Proposition~\ref{prop:saturation-JA}.
		
		(ii)$\Rightarrow$(i) By Lemma~\ref{lem:integral-basis-LA}, $\Lambda_A$ is saturated, so $I_{\Lambda_A}$ is prime.
		
		(i)$\Rightarrow$(iii) As in the proof of Proposition~\ref{prop:turning-relation-nonprime}, $J_A$ contains no monomial, so $X_L\notin J_A$. Since $J_A$ is prime, we get $J_A:X_L^n=J_A$ for every $n\geq0$, and hence $J_A:X_L^\infty=J_A$.
	\end{proof}
	
	\section{Proof of the sufficiency}\label{sec:sufficiency}
	
	Throughout this section, $L$ is a thin distributive lattice of rank $m+1$ and $A\subseteq[m]$. By Propositions~\ref{prop:turning-relation-nonprime} and~\ref{prop:neighboring-relations-nonprime}, if $J_A$ is prime, then $A$ is admissible. In this section we prove the converse and, in view of Theorem~\ref{thm:prime-iff-lattice-ideal}, we have to show that $J_A=I_{\Lambda_A}$ whenever $A$ is admissible.
	
	A binomial $f\in I_{\Lambda_A}$ is called \emph{redundant} if it belongs to the ideal generated by the binomials of $I_{\Lambda_A}$ of degree strictly smaller than $\deg(f)$, and \emph{irredundant} otherwise. The key result of this section is the following.
	
	\begin{Proposition}\label{prop:main}
		Let $L$ be a thin distributive lattice of rank $m+1$, and assume that $A\subseteq[m]$ is admissible. Then the following statements hold:
		\begin{enumerate}[(i)]
			\item every binomial of degree one in $I_{\Lambda_A}$ belongs to $(\ell_i:i\in A)$, and every binomial of degree two in $I_{\Lambda_A}$ belongs to $J_A$;
			\item every irredundant binomial in $I_{\Lambda_A}$ has degree at most two.
		\end{enumerate}
	\end{Proposition}
	
	The main tool to control the binomials of $I_{\Lambda_A}$ is given by the following lemma.
	
	\begin{Lemma}\label{lem:weights}
		Let $\theta:L\to\ZZ$ be a function such that $\theta(a_i)+\theta(b_i)=\theta(c_i)+\theta(d_i)$ for every $i\in[m]$, and $\theta(c_i)=\theta(d_i)$ for every $i\in A$. For $\uu\in\NN^L$, set $\theta(\uu)=\sum_{w\in L}\uu(w)\theta(w)$. If $\xx^\uu-\xx^\vv\in I_{\Lambda_A}$, then $\theta(\uu)=\theta(\vv)$.
	\end{Lemma}
	
	\begin{proof}
		Let $\Theta:\ZZ^L\to\ZZ$ be the homomorphism of abelian groups defined on the canonical basis by $\Theta(\ee_w)=\theta(w)$ for every $w\in L$, so that $\Theta(\uu)=\theta(\uu)$ for every $\uu\in\NN^L$. For every $i\in[m]$, the first assumption on $\theta$ gives $\Theta(\etab_i)=\theta(a_i)+\theta(b_i)-\theta(c_i)-\theta(d_i)=0$, and for every $i\in A$ the second one gives $\Theta(\rhob_i)=\theta(c_i)-\theta(d_i)=0$. Since $\Lambda_A$ is generated by these elements, we get $\Lambda_A\subseteq\ker\Theta$. Now, if $\xx^\uu-\xx^\vv\in I_{\Lambda_A}$, then $\uu-\vv\in\Lambda_A$, and therefore $\theta(\uu)-\theta(\vv)=\Theta(\uu-\vv)=0$.
	\end{proof}
	
	A function $\theta$ as in Lemma~\ref{lem:weights} is called an \emph{$A$-weight}. In other words, $\theta$ is an $A$-weight if, on every elementary diamond, the sum of its values at the least and at the greatest element equals the sum of its values at the two elements of the middle rank, and, for every $i\in A$, these two values coincide.
	
	Let $f$ be a binomial of $I_{\Lambda_A}$ of degree $d$. We say that $f$ is \emph{good} if one of the following holds:
	\begin{itemize}
		\item $d=1$ and $f\in(\ell_i:i\in A)$;
		\item $d=2$ and $f\in J_A$;
		\item $d\geq3$ and $f$ is redundant.
	\end{itemize}
	We also agree that $0$ is good. Clearly, Proposition~\ref{prop:main} states that every binomial of $I_{\Lambda_A}$ is good. The next lemma is the analogue of \cite[Lemma~2.2]{CistoNavarra}.
	
	\begin{Lemma}\label{lem:moves}
		Let $f=\xx^\uu-\xx^\vv\in I_{\Lambda_A}$ be a binomial of degree $d$.
		\begin{enumerate}[(a)]
			\item Let $g=\xx^{\mathbf r}-\xx^{\mathbf s}$, where either $g=\pm\ell_i$ for some $i\in A$, or $d\geq2$ and $g$ is, up to sign, a join-meet binomial of $L$. If $\xx^{\mathbf r}$ divides $\xx^\uu$ and $\uu'=\uu-\mathbf r+\mathbf s$, then $f'=\xx^{\uu'}-\xx^\vv$ is either zero or a binomial of $I_{\Lambda_A}$ of degree $d$, and $f$ is good if and only if $f'$ is good. The same holds if the roles of $\xx^\uu$ and $\xx^\vv$ are exchanged. We say that $f'$ is obtained from $f$ by a \emph{move}.
			\item If some variable $x_w$ divides both $\xx^\uu$ and $\xx^\vv$, then $f=x_wf_1$, where $f_1$ is a binomial of $I_{\Lambda_A}$ of degree $d-1$. If $d\geq3$, or $d=2$ and $f_1$ is good, then $f$ is good.
		\end{enumerate}
	\end{Lemma}
	
	\begin{proof}
		(a) We have $f-f'=\xx^{\uu-\mathbf r}g\in J_A\subseteq I_{\Lambda_A}$ by Proposition~\ref{prop:saturation-JA}, so $f'\in I_{\Lambda_A}$, and $\deg\xx^{\uu'}=\deg\xx^\uu=d$ because $g$ is homogeneous. If $d=1$, then $g=\pm\ell_i$ and $f-f'\in(\ell_i:i\in A)$. If $d=2$, then $f-f'\in J_A$. If $d\geq3$, then $f-f'$ is a multiple of $g$, which is a binomial of $I_{\Lambda_A}$ of degree at most $2<d$. In each case, $f$ is good if and only if $f'$ is good.
		
		(b) By Lemma~\ref{lem:integral-basis-LA}, $\Lambda_A$ is saturated, so $I_{\Lambda_A}$ is a prime ideal. Moreover, $I_{\Lambda_A}$ contains no monomial: indeed, it is contained in the kernel of the homomorphism $S_L\to K$ sending every variable to $1$, which maps every monomial to $1$. Hence $x_w\notin I_{\Lambda_A}$, and $x_wf_1=f\in I_{\Lambda_A}$ gives $f_1\in I_{\Lambda_A}$; clearly, $f_1$ is a binomial of degree $d-1$. If $d\geq3$, then $f$ is a multiple of the binomial $f_1\in I_{\Lambda_A}$ of degree $d-1$, so it is redundant. If $d=2$ and $f_1\in(\ell_i:i\in A)$, then $f\in J_A$.
	\end{proof}
	
	We now fix the setting in which we perform the induction on the rank of $L$. Recall that $\cT(L)\subseteq\{2,\ldots,m-1\}$.
	
	\begin{Setup}\label{setup}
		Let $m\geq2$, and assume that $A$ is admissible. Set $r=\max\cT(L)$ if $\cT(L)\neq\emptyset$, and $r=1$ otherwise. Then $1\leq r\leq m-1$, and we set $n=m-r+1$, so that $n\geq2$. No turning rank $k$ satisfies $r<k<m$, so $[a_r,b_m]=D_r\cup\cdots\cup D_m$ is a ladder, and we denote by $\alpha_j$ and $\gamma_j$, $0\leq j\leq n$, its elements denoted by $\lambda_j$ and $\mu_j$ in Discussion~\ref{lem:structure}. Thus
		\[
		B=[a_r,1_L]=\{\alpha_0,\ldots,\alpha_n,\gamma_0,\ldots,\gamma_n\}\cong\Chain{n+1}\times\Chain2,
		\]
		$\alpha_i\leq\gamma_j$ if and only if $i\leq j$, and $\gamma_i\nleq\alpha_j$ for all $i,j$. For $j\in[n]$ we set $B_j=D_{r+j-1}$, so that
		\[
		B_j=[\alpha_{j-1},\gamma_j]=\{\alpha_{j-1},\gamma_{j-1},\alpha_j,\gamma_j\},\qquad\{c_{r+j-1},d_{r+j-1}\}=\{\gamma_{j-1},\alpha_j\}.
		\]
		In particular, the linear binomial attached to $B_j$ is $\ell_{r+j-1}=\pm(x_{\gamma_{j-1}}-x_{\alpha_j})$. Moreover, $\alpha_j=a_{r+j}$ for $0\leq j\leq n-1$ and $\gamma_j=b_{r+j-1}$ for $1\leq j\leq n$; in particular, $\gamma_n=b_m=1_L$. We also denote by $\cA=\{D_i:i\in A\}$ the set of elementary diamonds whose rank belongs to $A$; thus $B_j\in\cA$ if and only if $r+j-1\in A$.
		
		We set $p=\gamma_n=1_L$ and $q=\alpha_n$. Let $k_0$ be the least integer in $\{0,\ldots,n\}$ such that $B_j\in\cA$ for every $k_0<j\leq n$; in other words, $B_{k_0+1},\ldots,B_n$ is the longest sequence of consecutive diamonds of $\cA$ ending with $B_n=D_m$. In particular, $k_0=n$ if $m\notin A$, and $B_{k_0}\notin\cA$ if $k_0\geq1$. Finally, we set
		\[
		t=n-k_0,\qquad F=\{\alpha_j,\gamma_j:k_0\leq j\leq n\}.
		\]
		Thus $t\geq1$ if and only if $m\in A$, and $F=\{p,q\}$ if $m\notin A$. See Figure~\ref{fig:setting}.
	\end{Setup}
	
	\begin{figure}[htbp]
		\centering
		\begin{tikzpicture}[x=0.62cm,y=0.62cm,vF/.style={circle,fill=black,inner sep=1.3pt},vB/.style={circle,draw=black,fill=white,inner sep=1.3pt},lab/.style={font=\scriptsize}]
			\begin{scope}[cm={1,1,-1,1,(0,0)}]
				\foreach \k in {2,4,5,6} \fill[gray!25] (\k-1,0) rectangle (\k,1);
				\draw[dashed] (-1,-2) rectangle (0,-1);
				\draw (0,-2) rectangle (1,0);
				\draw (0,-1) -- (1,-1);
				\node[lab] at (0.5,-0.5) {$E_1$};
				\node[lab] at (0.5,-1.5) {$E_2$};
				\foreach \k in {1,...,6} \draw (\k-1,0) rectangle (\k,1);
				\foreach \k in {1,...,6} \node[lab] at (\k-0.5,0.5) {$B_{\k}$};
				\foreach \k in {0,1,2} {\node[vB] at (\k,0) {}; \node[vB] at (\k,1) {};}
				\foreach \k in {3,...,6} {\node[vF] at (\k,0) {}; \node[vF] at (\k,1) {};}
				\foreach \k in {0,...,5} \node[lab,above left] at (\k,1) {$\gamma_{\k}$};
				\node[lab,above left] at (6,1) {$\gamma_6=p$};
				\node[lab,left] at (0,0) {$\alpha_0$};
				\node[lab,right] at (1,0) {$\alpha_1$};
				\foreach \k in {2,...,5} \node[lab,below right] at (\k,0) {$\alpha_{\k}$};
				\node[lab,below right] at (6,0) {$\alpha_6=q$};
			\end{scope}
		\end{tikzpicture}
		\caption{The setting of Setup~\ref{setup} with $n=6$. The diamonds of $\cA$ are shaded, so that $k_0=3$, and the elements of $F$ are the black dots. The diamonds $E_1=D_{r-1}$ and $E_2=D_{r-2}$ will appear in the proof of Proposition~\ref{prop:main}.}
		\label{fig:setting}
	\end{figure}
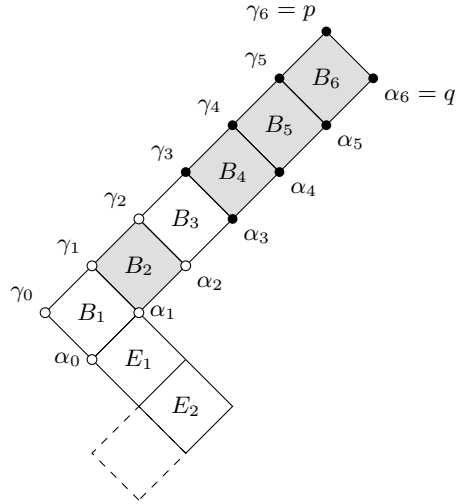
	
	\begin{Lemma}\label{lem:setup}
		In the setting of Setup~\ref{setup}, the following hold.
		\begin{enumerate}[(S1)]
			\item The elements $p$ and $q$ belong to no elementary diamond other than $D_m=B_n$, of which they are the greatest element and an element of rank $m$, respectively. Moreover, $L'=[0_L,b_{m-1}]$ is a thin distributive lattice of rank $m$, $L=L'\sqcup\{p,q\}$, the elementary diamonds of $L'$ are $D_1,\ldots,D_{m-1}$, the join-meet binomials of $L'$ are join-meet binomials of $L$, and $\cT(L')\subseteq\cT(L)$.
			\item If $r>1$, then $r\in\cT(L)$, $B_1\notin\cA$ and $k_0\geq1$. Moreover, the greatest element of $E_1=D_{r-1}$ is $\alpha_1$, and $D_i\cap B\subseteq\{\alpha_0,\alpha_1\}$ for every $i<r$.
			\item For $0\leq k<j\leq n$, the elements $\gamma_k$ and $\alpha_j$ are incomparable, $\gamma_k\wedge\alpha_j=\alpha_k$ and $\gamma_k\vee\alpha_j=\gamma_j$. In particular, $g_{k,j}=x_{\alpha_k}x_{\gamma_j}-x_{\gamma_k}x_{\alpha_j}$ is, up to sign, a join-meet binomial of $L$.
		\end{enumerate}
	\end{Lemma}
	
	\begin{proof}
		(S1) The elements of $D_i$ have rank between $i-1$ and $i+1$. Since $p=1_L$ has rank $m+1$, it belongs only to $D_m$. The element $q$ has rank $m$, so it can only belong to $D_{m-1}$ and $D_m$; since the only element of $D_{m-1}$ of rank $m$ is $b_{m-1}=\gamma_{n-1}$, and $\gamma_{n-1}\neq\alpha_n=q$, the element $q$ belongs only to $D_m$.
		
		For $1\leq i\leq m-2$ we have $b_i\in\{c_{i+1},d_{i+1}\}$ by Discussion~\ref{lem:structure}, so $b_i\leq b_{i+1}$; hence $c_i,d_i\leq b_i\leq b_{m-1}$ for every $i\in[m-1]$. Therefore $L'=[0_L,b_{m-1}]$ contains $0_L$ and all the elements of rank $1,\ldots,m-1$. On the other hand, the only element of $L'$ of rank at least $m$ is $b_{m-1}$, since $q$ is incomparable with $b_{m-1}$ and $1_L\nleq b_{m-1}$. Hence $L=L'\sqcup\{p,q\}$. By Remark~\ref{lem:basic}(c), $L'$ is a distributive lattice with the operations of $L$, and its rank function is the restriction of $\rank_L$, by Remark~\ref{lem:basic}(a). Thus $L'$ has rank $m$ and exactly two elements of rank $i$ for every $i\in[m-1]$; since $m-1\geq1$, it is a thin distributive lattice. Its elements of rank $i$ are $c_i$ and $d_i$, so its elementary diamonds are $D_1,\ldots,D_{m-1}$, and an integer $k\in\{2,\ldots,m-2\}$ is a turning rank of $L'$ if and only if $b_{k-1}=a_{k+1}$, that is, if and only if $k\in\cT(L)$. Finally, two elements of $L'$ are incomparable in $L'$ if and only if they are incomparable in $L$, and their meet and join are the same in $L'$ and in $L$.
		
		(S2) If $r>1$, then $r=\max\cT(L)\in\cT(L)$, so $B_1=D_r\notin\cA$ by condition (i) of Definition~\ref{Def admissible}, and hence $k_0\geq1$. Since $r\in\cT(L)$ and $n\geq2$, the greatest element of $D_{r-1}$ is $b_{r-1}=a_{r+1}=\alpha_1$. Let $i<r$. The elements of $D_i$ have rank at most $i+1\leq r$, and the elements of $B$ of rank at most $r$ are $\alpha_0$, of rank $r-1$, and $\gamma_0,\alpha_1$, of rank $r$. If $\gamma_0\in D_i$, then $\rank_L(\gamma_0)=r$ forces $i=r-1$ and $\gamma_0=b_{r-1}=\alpha_1$, which is impossible. Therefore $D_i\cap B\subseteq\{\alpha_0,\alpha_1\}$.
		
		(S3) This follows from Discussion~\ref{lem:structure}, since $\gamma_k$ and $\alpha_j$ correspond to $(k,1)$ and $(j,0)$ of $\Chain{n+1}\times\Chain2$, respectively, and meets and joins in a ladder are computed componentwise.
	\end{proof}
	
	We introduce two $A$-weights. Their values on $B$ in the case of Figure~\ref{fig:setting} are displayed in Figure~\ref{fig:omega}.
	
	\begin{figure}[htbp]
		\centering
		\begin{tikzpicture}[x=0.62cm,y=0.62cm,vF/.style={circle,fill=black,inner sep=1.3pt},vB/.style={circle,draw=black,fill=white,inner sep=1.3pt},val/.style={font=\footnotesize}]
			\begin{scope}[cm={1,1,-1,1,(0,0)}]
				\foreach \k in {2,4,5,6} \fill[gray!25] (\k-1,0) rectangle (\k,1);
				\foreach \k in {2,4,5,6} \draw[densely dashed,gray] (\k-1,1) -- (\k,0);
				\draw[dashed] (-1,-2) rectangle (0,-1);
				\draw (0,-2) rectangle (1,0);
				\draw (0,-1) -- (1,-1);
				\foreach \k in {1,...,6} \draw (\k-1,0) rectangle (\k,1);
				\foreach \k in {0,1,2} {\node[vB] at (\k,0) {}; \node[vB] at (\k,1) {};}
				\foreach \k in {3,...,6} {\node[vF] at (\k,0) {}; \node[vF] at (\k,1) {};}
				\foreach \v [count=\k from 0] in {1,1,2,1,2,3,4} \node[val,above left] at (\k,1) {$\v$};
				\node[val,left] at (0,0) {$0$};
				\node[val,right] at (1,0) {$0$};
				\foreach \v [count=\k from 2] in {1,0,1,2,3} \node[val,below right] at (\k,0) {$\v$};
			\end{scope}
		\end{tikzpicture}
		\caption{The values of the $A$-weight $\omega$ of Lemma~\ref{lem:weights-setting}(a) in the setting of Figure~\ref{fig:setting}; $\omega$ vanishes at all the other elements of $L$. The dashed segments join the two elements of the middle rank of each diamond of $\cA$, at which $\omega$ takes the same value.}
		\label{fig:omega}
	\end{figure}
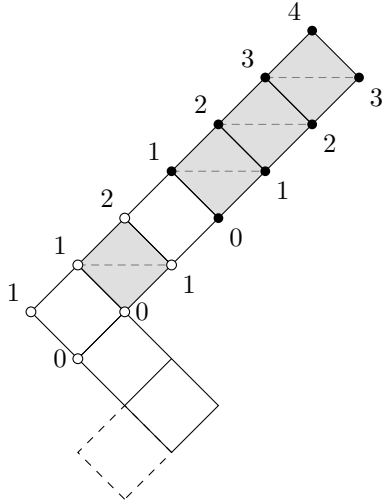
	
	\begin{Lemma}\label{lem:weights-setting}
		In the setting of Setup~\ref{setup}, let $\zeta(0)=0$ and, for $j\in[n]$, let $\zeta(j)=\zeta(j-1)+1$ if $B_j\in\cA$, and $\zeta(j)=0$ otherwise; thus $\zeta(j)$ is the number of consecutive diamonds of $\cA$ ending with $B_j$. Then the following hold.
		\begin{enumerate}[(a)]
			\item The function $\omega:L\to\ZZ$ defined by $\omega(\alpha_j)=\zeta(j)$ and $\omega(\gamma_j)=\zeta(j)+1$ for $0\leq j\leq n$, and $\omega(w)=0$ for $w\notin B$, is a non-negative $A$-weight.
			\item If $k_0\neq1$ or $r=1$, then the characteristic function $\delta_F$ of $F$, defined by $\delta_F(w)=1$ if $w\in F$ and $\delta_F(w)=0$ otherwise, is an $A$-weight.
			\item For $k_0\leq j\leq n$ one has $\omega(\alpha_j)=j-k_0$ and $\omega(\gamma_j)=j-k_0+1$; in particular, $\omega(p)=t+1$ and $\omega(w)\leq t$ for every $w\in F\setminus\{p\}$. Moreover, if $\omega(w)\geq1$, then either $w=\gamma_j$ for some $0\leq j\leq n$, or $w=\alpha_j$ for some $j\in[n]$ with $B_j\in\cA$ and $\omega(\gamma_{j-1})=\omega(\alpha_j)$.
		\end{enumerate}
	\end{Lemma}
	
	\begin{proof}
		(a) Clearly $\omega\geq0$. Let $j\in[n]$. The least and the greatest element of $B_j$ are $\alpha_{j-1}$ and $\gamma_j$, and its elements of the middle rank are $\gamma_{j-1}$ and $\alpha_j$. Both $\omega(\alpha_{j-1})+\omega(\gamma_j)$ and $\omega(\gamma_{j-1})+\omega(\alpha_j)$ are equal to $\zeta(j-1)+\zeta(j)+1$, and if $B_j\in\cA$, then $\omega(\gamma_{j-1})=\zeta(j-1)+1=\zeta(j)=\omega(\alpha_j)$. If $r=1$, there are no other elementary diamonds. If $r>1$, then $B_1\notin\cA$ by (S2), so $\zeta(1)=0$ and $\omega$ vanishes at $\alpha_0$ and $\alpha_1$. Hence, again by (S2), $\omega$ vanishes on $D_i$ for every $i<r$, and the conditions of Lemma~\ref{lem:weights} hold trivially on these diamonds.
		
		(b) Let $j\in[n]$. The diamond $B_j$ has no element in $F$ if $j<k_0$, all its elements in $F$ if $j>k_0$, and exactly its greatest element $\gamma_{k_0}$ and one element of the middle rank, namely $\alpha_{k_0}$, in $F$ if $j=k_0\geq1$. In each case, the sum of the values of $\delta_F$ at the least and at the greatest element of $B_j$ equals the sum of its values at the two elements of the middle rank. Moreover, if $B_j\in\cA$, then $j\neq k_0$, so $\gamma_{j-1}$ and $\alpha_j$ are either both in $F$ or both outside $F$. If $r=1$ there are no other elementary diamonds. If $r>1$, then $k_0\geq1$ by (S2), and hence $k_0\geq2$ by our assumption, so that $\alpha_0,\alpha_1\notin F$; thus $\delta_F$ vanishes on $D_i$ for every $i<r$ by (S2).
		
		(c) We first compute $\zeta(j)$ for $k_0\leq j\leq n$. We have $\zeta(k_0)=0$: this holds by definition if $k_0=0$, while if $k_0\geq1$, then $B_{k_0}\notin\cA$ by the choice of $k_0$, and so $\zeta(k_0)=0$ by the definition of $\zeta$. Moreover, for $k_0<j\leq n$ we have $B_j\in\cA$, hence $\zeta(j)=\zeta(j-1)+1$. By induction on $j$, it follows that $\zeta(j)=j-k_0$ for every $k_0\leq j\leq n$, and therefore $\omega(\alpha_j)=j-k_0$ and $\omega(\gamma_j)=j-k_0+1$ for $k_0\leq j\leq n$. In particular, $\omega(p)=\omega(\gamma_n)=n-k_0+1=t+1$. Let $w\in F\setminus\{p\}$. If $w=\alpha_j$ for some $k_0\leq j\leq n$, then $\omega(w)=j-k_0\leq n-k_0=t$. Otherwise $w=\gamma_j$ for some $k_0\leq j\leq n-1$, and $\omega(w)=j-k_0+1\leq n-k_0=t$. Hence $\omega(w)\leq t$ for every $w\in F\setminus\{p\}$.
		
		Finally, let $w\in L$ with $\omega(w)\geq1$. Since $\omega$ vanishes outside $B$, we have $w=\alpha_j$ or $w=\gamma_j$ for some $0\leq j\leq n$. Assume that $w=\alpha_j$. Then $\zeta(j)=\omega(\alpha_j)\geq1$. Since $\zeta(0)=0$, we have $j\geq1$, and since $\zeta(j)=0$ whenever $B_j\notin\cA$, we get $B_j\in\cA$. Hence $\zeta(j)=\zeta(j-1)+1$ by the definition of $\zeta$, and so $\omega(\gamma_{j-1})=\zeta(j-1)+1=\zeta(j)=\omega(\alpha_j)$.
	\end{proof}
	
	The following lemma is the heart of the proof: it shows how to move, inside $B$, a variable indexed by an element of $F$ up to $x_p$.
	
	\begin{Lemma}\label{lem:transport}
		In the setting of Setup~\ref{setup}, let $f=\xx^\uu-\xx^\vv\in I_{\Lambda_A}$ be a binomial of degree $d$ such that $x_p\mid\xx^\uu$, $x_p\nmid\xx^\vv$, and $x_w\mid\xx^\vv$ for some $w\in F$. If every binomial of $I_{\Lambda_A}$ of degree less than $d$ is good, then $f$ is good.
	\end{Lemma}
	
	\begin{proof}
		Let $a$ be the maximum of $\omega(w)$ over the elements $w\in F$ such that $x_w\mid\xx^\vv$. Since $x_p\nmid\xx^\vv$, we have $a\leq t$ by Lemma~\ref{lem:weights-setting}(c), and we proceed by induction on $t-a$. Set $j=k_0+a$, so that $k_0\leq j\leq n$. By Lemma~\ref{lem:weights-setting}(c), for $k_0\leq k\leq n$ we have $\omega(\alpha_k)=a$ if and only if $k=j$, and $\omega(\gamma_k)=a$ if and only if $k=j-1$, the latter being possible only when $j-1\geq k_0$, that is, $a\geq1$. Hence the elements of $F$ at which $\omega$ takes the value $a$ are $\alpha_j$ and, if $a\geq1$, $\gamma_{j-1}$. By the choice of $a$, the monomial $\xx^\vv$ is divisible by $x_{\alpha_j}$ or by $x_{\gamma_{j-1}}$.
		
		We claim that we may assume $x_{\alpha_j}\mid\xx^\vv$. Suppose that $x_{\alpha_j}\nmid\xx^\vv$, so that $a\geq1$ and $x_{\gamma_{j-1}}\mid\xx^\vv$. Since $j>k_0$, we have $B_j\in\cA$, and the linear binomial attached to $B_j$ is $\pm(x_{\gamma_{j-1}}-x_{\alpha_j})$. We apply Lemma~\ref{lem:moves}(a) to $\xx^\vv$ with $g=x_{\gamma_{j-1}}-x_{\alpha_j}$, that is, we replace one factor $x_{\gamma_{j-1}}$ of $\xx^\vv$ by $x_{\alpha_j}$. We obtain $f'=\xx^\uu-\xx^{\vv'}$, with $\vv'=\vv-\ee_{\gamma_{j-1}}+\ee_{\alpha_j}$, which is either zero or a binomial of $I_{\Lambda_A}$ of degree $d$, and $f$ is good if and only if $f'$ is good. If $f'=0$, then $f'$ is good and we are done. Otherwise, $f'$ satisfies the same assumptions as $f$, with the same value of $a$: indeed, $\xx^\uu$ has not changed, $x_p\nmid\xx^{\vv'}$ since $\alpha_j\neq p$, and $\xx^{\vv'}$ is divisible by $x_{\alpha_j}$, with $\alpha_j\in F$, while every other element of $F$ whose variable divides $\xx^{\vv'}$ already occurs in $\xx^\vv$. Replacing $f$ by $f'$, we may therefore assume that $x_{\alpha_j}\mid\xx^\vv$.
		
		By Lemmas~\ref{lem:weights} and~\ref{lem:weights-setting}(a), we have $\omega(\vv)=\omega(\uu)\geq\omega(p)=t+1$. Hence $\xx^\vv=x_{\alpha_j}\xx^{\vv'}$ with $\omega(\vv')=\omega(\vv)-a\geq t+1-a\geq1$; in particular, $d\geq2$. Since $\omega$ is non-negative and $\omega(\vv')=\sum_{w\in L}\vv'(w)\,\omega(w)\geq1$, there exists $w'\in L$ such that $x_{w'}$ divides $\xx^{\vv'}$ and $\omega(w')\geq1$. By Lemma~\ref{lem:weights-setting}(c), either $w'=\gamma_k$ for some $0\leq k\leq n$, or $w'=\alpha_k$ for some $k\in[n]$ such that $B_k\in\cA$ and $\omega(\gamma_{k-1})=\omega(\alpha_k)$.
		
		We show that the second case can be reduced to the first one. Assume that $w'=\alpha_k$ with $B_k\in\cA$. Applying Lemma~\ref{lem:moves}(a) with $g=x_{\alpha_k}-x_{\gamma_{k-1}}$, we replace the factor $x_{\alpha_k}$ of $\xx^{\vv'}$ by $x_{\gamma_{k-1}}$, and we obtain $f'=\xx^\uu-x_{\alpha_j}\xx^{\vv''}$, where $\vv''=\vv'-\ee_{\alpha_k}+\ee_{\gamma_{k-1}}$. Then $f'$ is either zero, and hence good, or a binomial of $I_{\Lambda_A}$ of degree $d$; moreover, $f$ is good if and only if $f'$ is good. We now verify that $f'$ satisfies the same assumptions as $f$, with the same values of $a$ and $j$. The first monomial $\xx^\uu$ is unchanged, while the second monomial $x_{\alpha_j}\xx^{\vv''}$ is divisible by $x_{\alpha_j}$ and is not divisible by $x_p$, because $\gamma_{k-1}\neq p$. Furthermore, if $\alpha_k\in F$, then $k\geq k_0$; since $k\geq1$ and $B_{k_0}\notin\cA$ whenever $k_0\geq1$, we get $k>k_0$, hence $\gamma_{k-1}\in F$ and $\omega(\gamma_{k-1})=\omega(\alpha_k)\leq a$. If $\alpha_k\notin F$, then $k<k_0$, and so $\gamma_{k-1}\notin F$. In both cases, the maximum of $\omega$ over the elements of $F$ whose variables divide $x_{\alpha_j}\xx^{\vv''}$ is equal to $a$ and is attained at $\alpha_j$. Finally, $\omega(\vv'')=\omega(\vv')$, and $x_{\gamma_{k-1}}$ divides $\xx^{\vv''}$, with $\omega(\gamma_{k-1})\geq1$. Therefore, replacing $f$ by $f'$, we may assume that $w'=\gamma_k$ for some $0\leq k\leq n$.
		
		Then $k\leq n-1$, since $x_p\nmid\xx^\vv$, and $k<j$: this is clear if $\gamma_k\notin F$, because then $k<k_0\leq j$, while if $\gamma_k\in F$, then $k-k_0+1=\omega(\gamma_k)\leq a=j-k_0$ by the choice of $a$. By (S3), $g=x_{\gamma_k}x_{\alpha_j}-x_{\alpha_k}x_{\gamma_j}=-g_{k,j}$ is, up to sign, a join-meet binomial of $L$. The monomial $x_{\gamma_k}x_{\alpha_j}$ divides $\xx^\vv=x_{\alpha_j}\xx^{\vv'}$, because $x_{\gamma_k}$ divides $\xx^{\vv'}$, and $d\geq2$. Applying Lemma~\ref{lem:moves}(a) with this $g$, we replace $x_{\gamma_k}x_{\alpha_j}$ by $x_{\alpha_k}x_{\gamma_j}$ in $\xx^\vv$, and we obtain $f'=\xx^\uu-\xx^{\vv^*}$, where $\vv^*=\vv-\ee_{\gamma_k}-\ee_{\alpha_j}+\ee_{\alpha_k}+\ee_{\gamma_j}$. Then $f'$ is either zero, and hence good, or a binomial of $I_{\Lambda_A}$ of degree $d$, and $f$ is good if and only if $f'$ is good. We may assume $f'\neq0$, and we distinguish two cases.
		
		Assume first that $a=t$. Then $j=k_0+t=n$, and so $\gamma_j=p$. Hence $x_p$ divides both $\xx^\uu$ and $\xx^{\vv^*}$, and $f'=x_pf_1$ for some binomial $f_1\in I_{\Lambda_A}$ of degree $d-1$. If $d\geq3$, then $f'$ is good by Lemma~\ref{lem:moves}(b). If $d=2$, then $f_1$ has degree one, so it is good by our assumption on the binomials of degree less than $d$, and $f'$ is good again by Lemma~\ref{lem:moves}(b). This settles the case $t-a=0$.
		
		Assume now that $a<t$. Then $j<n$, so that $\gamma_j\neq p$, and also $\alpha_k\neq p$. Since $x_p\nmid\xx^\vv$, it follows that $x_p\nmid\xx^{\vv^*}$. Moreover, $\xx^{\vv^*}$ is divisible by $x_{\gamma_j}$, where $\gamma_j\in F$ and $\omega(\gamma_j)=j-k_0+1=a+1$ by Lemma~\ref{lem:weights-setting}(c). Let $w\in F\setminus\{\gamma_j\}$ be such that $x_w$ divides $\xx^{\vv^*}$. If $w\neq\alpha_k$, then $x_w$ divides $\xx^\vv$, and so $\omega(w)\leq a$ by the choice of $a$. If $w=\alpha_k$, then $\gamma_k\in F$, since both $\alpha_k$ and $\gamma_k$ belong to $F$ exactly when $k\geq k_0$, and $x_{\gamma_k}$ divides $\xx^\vv$; therefore $\omega(\alpha_k)=\omega(\gamma_k)-1\leq a-1$. Consequently, $f'$ satisfies the assumptions of the lemma, and the maximum of $\omega$ over the elements of $F$ whose variables divide $\xx^{\vv^*}$ is $a+1$. Since $t-(a+1)<t-a$, the binomial $f'$ is good by the inductive hypothesis, and hence so is $f$.
	\end{proof}
	
	The last ingredient allows us to pass from $L$ to the thin distributive lattice $L'$ of (S1).
	
	\begin{Lemma}\label{lem:restriction}
		In the setting of Setup~\ref{setup}, let $L'=[0_L,b_{m-1}]$ and $A'=A\cap[m-1]$, and denote by $J'_{A'}\subseteq S_{L'}$ and $\Lambda'_{A'}\subseteq\ZZ^{L'}$ the ideal and the lattice attached to $(L',A')$. Then $A'$ is admissible for $L'$ and $J'_{A'}\subseteq J_A$. Moreover, if $\uu,\vv\in\NN^{L'}$, then $\xx^\uu-\xx^\vv\in I_{\Lambda_A}$ if and only if $\xx^\uu-\xx^\vv\in I_{\Lambda'_{A'}}$.
	\end{Lemma}
	
	\begin{proof}
		By (S1), the turning ranks of $L'$ are turning ranks of $L$, so $A'$ is admissible for $L'$; moreover, the join-meet binomials of $L'$ are join-meet binomials of $L$, and so $J'_{A'}\subseteq J_A$. We regard $\ZZ^{L'}$ as a subgroup of $\ZZ^L$ in the natural way. Since the elementary diamonds of $L'$ are $D_1,\ldots,D_{m-1}$, the lattice $\Lambda'_{A'}$ is generated by $\etab_i$, $i\in[m-1]$, and $\rhob_i$, $i\in A'$, so $\Lambda'_{A'}\subseteq\Lambda_A$, and one implication follows. Conversely, assume that $\xx^\uu-\xx^\vv\in I_{\Lambda_A}$, and write $\uu-\vv=\sum_{i\in[m]}z_i\etab_i+\sum_{i\in A}z'_i\rhob_i$ with $z_i,z'_i\in\ZZ$. By (S1), only $\etab_m$ and $\rhob_m$ can have a non-zero coordinate at $p$ or $q$, and $\etab_m(p)=1$, $\etab_m(q)=-1$, $\rhob_m(p)=0$ and $\rhob_m(q)=\pm1$. Since $\uu-\vv$ vanishes at $p$ and $q$, we get $z_m=0$ and, if $m\in A$, $z'_m=0$. Therefore $\uu-\vv\in\Lambda'_{A'}$, that is, $\xx^\uu-\xx^\vv\in I_{\Lambda'_{A'}}$.
	\end{proof}
	
	\begin{proof}[Proof of Proposition~\ref{prop:main}]
		We proceed by induction on $m$.
		
		\emph{The case $m=1$.} Then $L=\{0_L,c_1,d_1,1_L\}\cong\Chain2\times\Chain2$. If $A=\emptyset$, then $J_A=(x_{0_L}x_{1_L}-x_{c_1}x_{d_1})$, while if $A=\{1\}$, then $S_L/J_A\cong K[x_{0_L},x_{1_L},x_{c_1}]/(x_{0_L}x_{1_L}-x_{c_1}^2)$, by identifying $x_{d_1}$ with $x_{c_1}$. In both cases $S_L/J_A$ is a domain, since $x_{0_L}x_{1_L}-x_{c_1}x_{d_1}$ and $x_{0_L}x_{1_L}-x_{c_1}^2$ are irreducible polynomials: they have degree one in $x_{0_L}$, and the coefficient $x_{1_L}$ of $x_{0_L}$ is coprime with the constant term. Hence $J_A$ is prime, and so $J_A=I_{\Lambda_A}$ by Theorem~\ref{thm:prime-iff-lattice-ideal}. Since $J_A$ is generated by binomials of $I_{\Lambda_A}$ of degree at most two, statement (ii) and the second part of (i) are clear. Finally, the degree one component of $J_A$ is $K\ell_1$ if $1\in A$, and zero otherwise, so every binomial of degree one of $I_{\Lambda_A}$ belongs to $(\ell_i:i\in A)$.
		
		\emph{The inductive step.} Let $m\geq2$, assume that the statement holds for thin distributive lattices of rank $m$, and place ourselves in the setting of Setup~\ref{setup}; let $L'$ and $A'$ be as in Lemma~\ref{lem:restriction}. We prove that every binomial $f=\xx^\uu-\xx^\vv\in I_{\Lambda_A}$ is good by induction on $d=\deg f$, assuming that all the binomials of $I_{\Lambda_A}$ of degree less than $d$ are good.
		
		We first make two reductions. Assume that $m\in A$. Then $B_n=D_m\in\cA$, and the linear binomial attached to $B_n$ is $\ell_m=\pm(x_{\gamma_{n-1}}-x_q)$. Applying Lemma~\ref{lem:moves}(a) with $g=x_q-x_{\gamma_{n-1}}$ as many times as necessary, we replace every factor $x_q$ of $\xx^\uu$ and of $\xx^\vv$ by $x_{\gamma_{n-1}}$. These moves do not affect the goodness of $f$, and they do not introduce the variable $x_p$, since $\gamma_{n-1}\neq p$. Hence, if $m\in A$, we may assume that $x_q$ divides neither $\xx^\uu$ nor $\xx^\vv$. Moreover, if $\xx^\uu$ and $\xx^\vv$ have a common variable, then $f$ is good by Lemma~\ref{lem:moves}(b), where for $d=2$ we use that the binomials of $I_{\Lambda_A}$ of degree one are good. Therefore, we may also assume that $\xx^\uu$ and $\xx^\vv$ have no common variable. We now distinguish two cases, according to whether $x_p$ divides one of the two monomials or not.
		
		\emph{Case 1: $x_p$ divides neither $\xx^\uu$ nor $\xx^\vv$.} We claim that $x_q$ divides neither of them as well. If $m\in A$, this holds by the first reduction. If $m\notin A$, then $k_0=n\geq2$, $t=0$ and $F=\{p,q\}$, so $\delta_F$ is an $A$-weight by Lemma~\ref{lem:weights-setting}(b). By Lemma~\ref{lem:weights}, we get $\uu(p)+\uu(q)=\vv(p)+\vv(q)$, that is, $\uu(q)=\vv(q)$, because $\uu(p)=\vv(p)=0$. Since $\xx^\uu$ and $\xx^\vv$ have no common variable, it follows that $\uu(q)=\vv(q)=0$, and the claim is proved. Since $L=L'\sqcup\{p,q\}$ by (S1), we obtain $f\in S_{L'}$, and hence $f\in I_{\Lambda'_{A'}}$ by Lemma~\ref{lem:restriction}. By the inductive hypothesis on $m$, $f$ is good with respect to $(L',A')$, hence with respect to $(L,A)$: if $d=1$, then $f\in(\ell_i:i\in A')\subseteq(\ell_i:i\in A)$; if $d=2$, then $f\in J'_{A'}\subseteq J_A$; if $d\geq3$, then $f$ lies in the ideal generated by the binomials of $I_{\Lambda'_{A'}}$ of degree less than $d$, which belong to $I_{\Lambda_A}$.
		
		\emph{Case 2: $x_p$ divides one of the two monomials.} Replacing $f$ by $-f$, we may assume that $x_p\mid\xx^\uu$ and $x_p\nmid\xx^\vv$. By Lemma~\ref{lem:transport}, it is enough to show that, after finitely many moves on $\xx^\vv$ which do not introduce the variable $x_p$, the monomial $\xx^\vv$ becomes divisible by $x_w$ for some $w\in F$. Note that each of these moves either produces zero, in which case $f$ is good, or produces a binomial of $I_{\Lambda_A}$ of degree $d$ which is good if and only if $f$ is good and still satisfies $x_p\mid\xx^\uu$ and $x_p\nmid\xx^\vv$.
		\begin{itemize}
			\item \emph{Subcase 2a: $k_0\neq1$ or $r=1$.} In this case no move is needed. Indeed, $\delta_F$ is an $A$-weight by Lemma~\ref{lem:weights-setting}(b), and so $\delta_F(\vv)=\delta_F(\uu)\geq\uu(p)\geq1$ by Lemma~\ref{lem:weights}. Hence $x_w\mid\xx^\vv$ for some $w\in F$.
			
			\item \emph{Subcase 2b: $k_0=1$ and $r>1$.} This is the only point where condition (ii) of Definition~\ref{Def admissible} is used. Here $\delta_F$ is not an $A$-weight: indeed, by (S2), the greatest element of $E_1=D_{r-1}$ is $\alpha_1\in F$, while the other three elements of $E_1$ do not belong to $F$, since $E_1\cap B\subseteq\{\alpha_0,\alpha_1\}$ and $\alpha_0\notin F$. We therefore replace $\delta_F$ by a different $A$-weight $\psi$, and we proceed in two steps.
			
			\emph{Step 1: construction of $\psi$.} By (S2), $r$ is a turning rank and $B_1=D_r\notin\cA$, and $B_2\in\cA$ since $k_0=1$, that is, $r+1\in A$. Hence $r-1\notin A$, that is, $E_1\notin\cA$, by condition (ii) of Definition~\ref{Def admissible}. Set $r'=\max\bigl(\cT(L)\cap\{2,\ldots,r-1\}\bigr)$ if this set is non-empty, and $r'=1$ otherwise, and let $h=r-r'\geq1$. No turning rank $k$ satisfies $r'<k<r$, so $[a_{r'},b_r]=D_{r'}\cup\cdots\cup D_r$ is a ladder, and we denote by $\beta_y$ and $\beta'_y$, for $-h\leq y\leq1$, its elements denoted by $\lambda_{y+h}$ and $\mu_{y+h}$ in Discussion~\ref{lem:structure}. Thus $B'=[a_{r'},b_r]=\{\beta_y,\beta'_y:-h\leq y\leq1\}\cong\Chain{h+2}\times\Chain2$, and setting $E_i=D_{r-i}$ for $0\leq i\leq h$, Discussion~\ref{lem:structure} gives
			\[
			E_i=[\beta_{-i},\beta'_{1-i}]=\{\beta_{-i},\beta'_{-i},\beta_{1-i},\beta'_{1-i}\},\qquad\{c_{r-i},d_{r-i}\}=\{\beta'_{-i},\beta_{1-i}\}.
			\]
			In particular, $E_0=B_1$. By Discussion~\ref{lem:structure} and (S2), we have $\beta_0=a_r=\alpha_0$, $\beta'_0=b_{r-1}=\alpha_1$ and $\beta'_1=b_r=\gamma_1$; since $\{\beta'_0,\beta_1\}=\{c_r,d_r\}=\{\gamma_0,\alpha_1\}$, we also get $\beta_1=\gamma_0$. Moreover, $B\cap B'=\{\alpha_0,\alpha_1,\gamma_0,\gamma_1\}$: indeed, these four elements belong to $B'$, the elements of $B'$ have rank at most $r+1$, the only element of $B'$ of rank $r+1$ is $b_r=\gamma_1$, and the elements of $B$ of rank at most $r+1$ are $\alpha_0,\gamma_0,\alpha_1,\gamma_1,\alpha_2$. Finally, assume that $r'>1$. Then $r'\in\cT(L)$, so $E_h=D_{r'}\notin\cA$ by condition (i) of Definition~\ref{Def admissible}, and $b_{r'-1}=a_{r'+1}=\beta_{1-h}$ by Discussion~\ref{lem:structure}. We claim that every elementary diamond $D_i$ with $i<r'$ meets $B\cup B'$ at most in $\beta_{-h}=a_{r'}$ and $\beta_{1-h}$. Indeed, the elements of $D_i$ have rank at most $i+1\leq r'$. The elements of $B'$ of rank at most $r'$ are $\beta_{-h}$, of rank $r'-1$, and $\beta'_{-h},\beta_{1-h}$, of rank $r'$, while the elements of $B$ have rank at least $r-1\geq r'$, and the only one of rank $r'$, if any, is $\alpha_0=\beta_0$, which happens exactly when $h=1$, so that $\beta_0=\beta_{1-h}$. Finally, if $\beta'_{-h}\in D_i$, then $\rank_L(\beta'_{-h})=r'$ forces $i=r'-1$ and $\beta'_{-h}=b_{r'-1}=\beta_{1-h}$, which is impossible. This proves the claim.
			
			Define $\zeta'(-h)=0$ and, for $-h<y\leq1$, let $\zeta'(y)=\zeta'(y-1)+1$ if $E_{1-y}\in\cA$, and $\zeta'(y)=0$ otherwise. Since $E_0,E_1\notin\cA$, we have $\zeta'(0)=\zeta'(1)=0$, and if $r'>1$, then $\zeta'(1-h)=0$, since $E_h\notin\cA$. Let $\psi:L\to\ZZ$ be defined by $\psi(\beta_y)=\zeta'(y)$ and $\psi(\beta'_y)=\zeta'(y)+1$ for $-h\leq y\leq1$, $\psi(\alpha_j)=\psi(\gamma_j)=1$ for $2\leq j\leq n$, and $\psi(w)=0$ otherwise; this is well defined, because $\alpha_j,\gamma_j\notin B'$ for $j\geq2$. Then $\psi\geq0$, $\psi(p)=1$, and $\psi$ coincides with $\delta_F$ on $B$, since $F=\{\alpha_j,\gamma_j:1\leq j\leq n\}$, $\psi(\alpha_0)=\psi(\beta_0)=0$, $\psi(\gamma_0)=\psi(\beta_1)=0$, $\psi(\alpha_1)=\psi(\beta'_0)=1$ and $\psi(\gamma_1)=\psi(\beta'_1)=1$.
			
			We show that $\psi$ is an $A$-weight, checking the conditions of Lemma~\ref{lem:weights} on every elementary diamond. On the diamonds $B_1,\ldots,B_n$ this follows from the proof of Lemma~\ref{lem:weights-setting}(b), since $\psi$ and $\delta_F$ coincide on $B$ and that part of the proof does not use the assumption on $k_0$. For a diamond $E_{1-y}$ with $-h<y\leq0$, whose least and greatest elements are $\beta_{y-1}$ and $\beta'_y$, both $\psi(\beta_{y-1})+\psi(\beta'_y)$ and $\psi(\beta'_{y-1})+\psi(\beta_y)$ are equal to $\zeta'(y-1)+\zeta'(y)+1$, and if $E_{1-y}\in\cA$, then $\psi(\beta_y)=\zeta'(y)=\zeta'(y-1)+1=\psi(\beta'_{y-1})$. Finally, if $r'>1$, every diamond $D_i$ with $i<r'$ meets $B\cup B'$ at most in $\beta_{-h}$ and $\beta_{1-h}$, where $\psi$ vanishes because $\zeta'(-h)=\zeta'(1-h)=0$; hence $\psi$ vanishes on $D_i$, and the conditions of Lemma~\ref{lem:weights} hold trivially on it. Therefore $\psi$ is an $A$-weight.
			
			\emph{Step 2: moving a variable indexed by $F$ into $\xx^\vv$.} If $\xx^\vv$ is already divisible by some $x_w$ with $w\in F$, there is nothing to prove, so assume that it is not. Since $k_0=1$, the element $\gamma_0$ is the only element outside $F$ at which $\omega$ is positive. By Lemmas~\ref{lem:weights} and~\ref{lem:weights-setting}(a), we have $\omega(\vv)=\omega(\uu)\geq\omega(p)\geq1$, and therefore $x_{\gamma_0}\mid\xx^\vv$. Similarly, by Step~1, $\psi(\vv)=\psi(\uu)\geq\psi(p)=1$, so $\xx^\vv$ is divisible by $x_w$ for some $w\notin F$ with $\psi(w)\geq1$. By the definition of $\psi$, either $w=\beta'_y$ with $-h\leq y\leq-1$, or $w=\beta_y$ with $-h<y\leq-1$ and $\zeta'(y)\geq1$. In the latter case $E_{1-y}\in\cA$, its linear binomial is $\pm(x_{\beta_y}-x_{\beta'_{y-1}})$, and the corresponding move replaces $x_{\beta_y}$ by $x_{\beta'_{y-1}}$ in $\xx^\vv$; this move does not introduce $x_p$ and does not remove $x_{\gamma_0}$. So we may assume $w=\beta'_y$ for some $-h\leq y\leq-1$. Now $\gamma_0=\beta_1$ and $\beta'_y$ are incomparable, $\beta'_y\wedge\beta_1=\beta_y$ and $\beta'_y\vee\beta_1=\beta'_1=\gamma_1$, by Discussion~\ref{lem:structure}. The move given by the corresponding join-meet binomial replaces $x_{\gamma_0}x_{\beta'_y}$ by $x_{\beta_y}x_{\gamma_1}$ in $\xx^\vv$. The resulting monomial is divisible by $x_{\gamma_1}$, with $\gamma_1\in F$, and it is not divisible by $x_p$, because $\gamma_1\neq p$, as $n\geq2$, and $\beta_y\neq p$. By Lemma~\ref{lem:transport}, $f$ is good.
		\end{itemize}

		In all cases $f$ is good. This completes the inductive step, and the proof of Proposition~\ref{prop:main}.
	\end{proof}
	
	\begin{Theorem}\label{thm:admissible}
		Let $L$ be a thin distributive lattice of rank $m+1$, and let $A\subseteq[m]$ be admissible. Then $J_A=I_{\Lambda_A}$. In particular, $J_A$ is a prime ideal and $\dim(S_L/J_A)=m+2-|A|$.
	\end{Theorem}
	
	\begin{proof}
		We have $J_A\subseteq I_{\Lambda_A}$ by Proposition~\ref{prop:saturation-JA}. We show by induction on $d\geq1$ that every binomial of $I_{\Lambda_A}$ of degree $d$ belongs to $J_A$. For $d\leq2$ this is Proposition~\ref{prop:main}(i). If $d\geq3$, then such a binomial is redundant by Proposition~\ref{prop:main}(ii), so it belongs to the ideal generated by the binomials of $I_{\Lambda_A}$ of degree less than $d$, which are contained in $J_A$ by induction. Since the lattice ideal $I_{\Lambda_A}$ is generated by binomials, we obtain $I_{\Lambda_A}\subseteq J_A$, and hence $J_A=I_{\Lambda_A}$. Since $\Lambda_A$ is saturated by Lemma~\ref{lem:integral-basis-LA}, $I_{\Lambda_A}$ is a prime ideal, and $\dim(S_L/J_A)=|L|-\rank_\ZZ\Lambda_A=(2m+2)-(m+|A|)$ by Lemma~\ref{lem:integral-basis-LA}; see \cite{HHO} for the dimension formula for lattice ideals.
	\end{proof}
	
	\begin{proof}[Proof of Theorem~\ref{Thm for primality}]
		If $J_A$ is prime, then $A$ is admissible by Propositions~\ref{prop:turning-relation-nonprime} and~\ref{prop:neighboring-relations-nonprime}. Conversely, if $A$ is admissible, then $J_A$ is prime by Theorem~\ref{thm:admissible}.
	\end{proof}

    	We conclude this section with some observations on $S_L/J_A$ for an arbitrary $A\subseteq[m]$, not necessarily admissible, which follow from well-known facts on Hibi rings.
	
	\begin{Remark}\label{rem:homological}\rm
		Let $A\subseteq[m]$. By \cite{Hibi}, $K[L]$ is a Cohen--Macaulay algebra with straightening laws on $L$. For every $ j \in [m]$, $\ell_j=x_{c_j}-x_{d_j}$ is the sum of the variables indexed by the elements of $L$ of rank $j$, while $x_{0_L}$ and $x_{1_L}$ correspond to the ranks $0$ and $m+1$. Hence, by \cite[Theorem~4.1]{Eis}, the residue classes of $x_{0_L},\ell_1,\ldots,\ell_m,x_{1_L}$ form a $K[L]$-regular sequence, and so do the linear forms $\ell_i$, with $i\in A$, since every permutation of a homogeneous regular sequence is a regular sequence \cite[Exercise~1.5.23]{BrunsHerzog}. It follows that:
		\begin{enumerate}
		  \item $S_L/J_A=K[L]/(\ell_i:i\in A)$ is Cohen--Macaulay by \cite[Theorem~2.1.3(a)]{BrunsHerzog};
		  \item the reduced Hlbert-Poincar\'{e} series of $S_L/J_A$ is
		  		$$\frac{W_L(t)}{(1-t)^{m+2-|A|}},$$
		  	where $W_L(t)$ is the so-called \emph{descent polynomial} of $P_L$ (see \cite[Chapter~3]{Stanley}), which represents the $h$-polynomial of $K[L]$;
		  	\item the regularity of $S_L/J_A$ is $m-\rank(P_L)$, which is the regularity of $K[L]$ (see  \cite[Theorem~1.1]{EHS}).
		\end{enumerate}
	\end{Remark}
	
	Since $J_A$ contains the linear forms $\ell_i$, $i\in A$, it is natural to study the resolution of $S_L/J_A$ after factoring them out. Let $T$ be the polynomial ring over $K$ in the variables $x_w$ with $w\in L\setminus\{d_i:i\in A\}$, and let $\sigma_A:S_L\to T$ be the surjective homomorphism of $K$-algebras defined by $\sigma_A(x_{d_i})=x_{c_i}$ for $i\in A$ and $\sigma_A(x_w)=x_w$ otherwise. Its kernel is $(\ell_i:i\in A)$, so that, setting $\overline J_A=\sigma_A(I_L)$, we have $S_L/J_A\cong T/\overline J_A$ as graded $K$-algebras.
	
	\begin{Corollary}\label{cor:linear-resolution}
		Let $L$ be a thin distributive lattice of rank $m+1$, and let $A\subseteq[m]$. Then the following conditions are equivalent:
		\begin{enumerate}[(i)]
			\item $\overline J_A$ has a linear resolution;
			\item $\rank(P_L)=m-1$;
			\item $L\cong\Chain{m+1}\times\Chain2$;
			\item $\cT(L)=\emptyset$.
		\end{enumerate}
		Consequently, $J_A$ is a prime ideal, $S_L/J_A$ is Cohen--Macaulay and $\overline J_A$ has a linear resolution if and only if $L\cong\Chain{m+1}\times\Chain2$. Moreover, in this case these three properties hold for every $A\subseteq[m]$.
	\end{Corollary}
	
	\begin{proof}
		Since $L$ is not a chain and $|P_L|=m+1$, we have $\rank(P_L)\leq m-1$, and so $\reg(T/\overline J_A)=\reg(S_L/J_A)=m-\rank(P_L)\geq1$ by Remark~\ref{rem:homological}. Since $\overline J_A$ is generated by quadrics, it has a linear resolution if and only if $\reg(T/\overline J_A)=1$, and this proves (i)$\Leftrightarrow$(ii). If $\rank(P_L)=m-1$, then $P_L=C\sqcup\{z\}$ for a chain $C=\{p_1<\cdots<p_m\}$, and $z$ is incomparable with every element of $C$: otherwise $\{p_1\}$ or $P_L\setminus\{p_m\}$ would be a poset ideal comparable with every poset ideal of $P_L$, that is, an apex of $L$ different from $0_L$ and $1_L$. Hence $L\cong\cJ(C)\times\cJ(\{z\})\cong\Chain{m+1}\times\Chain2$, and the converse is clear; this proves (ii)$\Leftrightarrow$(iii). By Discussion~\ref{lem:structure}, $\cT(L)=\emptyset$ if and only if $L$ is a single ladder, which gives (iii)$\Leftrightarrow$(iv). Finally, if $\cT(L)=\emptyset$, then every $A\subseteq[m]$ is admissible, and the last assertions follow from Theorem~\ref{Thm for primality} and Remark~\ref{rem:homological}.
	\end{proof}

	\section{A toric parametrization}\label{sec:toric}
	
	In this section we show that, when $A$ is admissible, $S_L/J_A$ is a toric ring, which is obtained by a simple modification of the parametrization of $K[L]$ given by Hibi \cite{Hibi}.
	
	Let $L$ be a thin distributive lattice of rank $m+1$. Consider the polynomial ring $\mathscr R_L=K[z,y_p:p\in P_L]$ and the homomorphism of $K$-algebras
	\[
	\varphi:S_L\longrightarrow\mathscr R_L,\qquad\varphi(x_\alpha)=z\prod_{p\in\ideal\alpha}y_p.
	\]
	By \cite{Hibi}, $\ker\varphi=I_L$, so that $K[L]$ is isomorphic to the affine semigroup ring $K[z\prod_{p\in\ideal\alpha}y_p:\alpha\in L]$. The homomorphism induced by $\varphi$ on the exponent groups is
	\[
	\betab:\ZZ^L\longrightarrow\Gamma_L=\ZZ\ee_z\oplus\bigoplus_{p\in P_L}\ZZ\ee_p,\qquad\betab(\ee_\alpha)=\ee_z+\sum_{p\in\ideal\alpha}\ee_p.
	\]
	Recall that, if $\phi$ is a homomorphism of $K$-algebras from $S_L$ to a polynomial ring which maps every variable to a monomial, then $\ker\phi$ is the lattice ideal of the kernel of the induced homomorphism of exponent groups; see, for instance, \cite[Chapter~3]{HHO}. In particular, $I_L=\ker\varphi=I_{\ker_\ZZ\betab}$.
	
	For every $i\in[m]$ we have $\ideal{a_i}=\ideal{c_i}\cap\ideal{d_i}$ and $|\ideal{c_i}|=|\ideal{d_i}|=|\ideal{a_i}|+1$ by Remark~\ref{lem:basic}(a) and Discussion~\ref{lem:structure}. Hence there exist unique elements $p_i,p'_i\in P_L$ such that
	\[
	\ideal{c_i}=\ideal{a_i}\sqcup\{p_i\},\qquad\ideal{d_i}=\ideal{a_i}\sqcup\{p'_i\},
	\]
	and $p_i\neq p'_i$, since $c_i\neq d_i$. In other words, $p_i$ is the only join-irreducible element of $L$ which is smaller than or equal to $c_i$ but not to $d_i$, and similarly for $p'_i$. In particular,
	\[
	\betab(\rhob_i)=\betab(\ee_{c_i})-\betab(\ee_{d_i})=\ee_{p_i}-\ee_{p'_i}.
	\]
	For $A\subseteq[m]$, let $\sim_A$ be the equivalence relation on $P_L$ generated by $p_i\sim_Ap'_i$ for $i\in A$, and denote by $[p]$ the equivalence class of $p\in P_L$. Let $\mathscr R_{L,A}=K[z,y_{[p]}:[p]\in P_L/{\sim_A}]$ and consider the homomorphism of $K$-algebras
	\[
	\varphi_A:S_L\longrightarrow\mathscr R_{L,A},\qquad\varphi_A(x_\alpha)=z\prod_{p\in\ideal\alpha}y_{[p]}.
	\]
	Thus $\varphi_A$ is obtained from $\varphi$ by identifying the variables $y_{p_i}$ and $y_{p'_i}$ for every $i\in A$, and $\varphi_\emptyset=\varphi$. We denote by $T_A=K[\varphi_A(x_\alpha):\alpha\in L]\subseteq\mathscr R_{L,A}$ its image, which is an affine semigroup ring.
	
	\begin{Lemma}\label{lem:kernel-beta}
		The homomorphism $\betab$ is surjective, and $\ker_\ZZ\betab=\Lambda_\emptyset$.
	\end{Lemma}
	
	\begin{proof}
		Since $\ideal{0_L}=\emptyset$, we have $\betab(\ee_{0_L})=\ee_z$. Let $p\in P_L$. The sets $U=\{p'\in P_L:p'\leq p\}$ and $U\setminus\{p\}$ are poset ideals of $P_L$, so by Birkhoff's theorem there exist $\alpha,\alpha'\in L$ with $\ideal\alpha=U$ and $\ideal{\alpha'}=U\setminus\{p\}$, and $\betab(\ee_\alpha)-\betab(\ee_{\alpha'})=\ee_p$. Hence $\betab$ is surjective.
		
		By \cite{Hibi}, $I_L=J_\emptyset$ is a prime ideal, so Theorem~\ref{thm:prime-iff-lattice-ideal} gives $I_L=I_{\Lambda_\emptyset}$. Therefore $I_{\ker_\ZZ\betab}=I_{\Lambda_\emptyset}$. Since a lattice is determined by its lattice ideal, because $\zz\in\Lambda$ if and only if $\xx^{\zz^+}-\xx^{\zz^-}\in I_\Lambda$ for every subgroup $\Lambda\subseteq\ZZ^L$, we conclude that $\ker_\ZZ\betab=\Lambda_\emptyset$.
	\end{proof}
	
	\begin{Proposition}\label{prop:toric-kernel}
		Let $L$ be a thin distributive lattice of rank $m+1$, and let $A\subseteq[m]$. Then
		\[
		J_A\subseteq\ker\varphi_A=I_{\Lambda_A}.
		\]
		In particular, $T_A\cong S_L/I_{\Lambda_A}$ and $\dim T_A=m+2-|A|$.
	\end{Proposition}
	
	\begin{proof}
		Let $\pi_A:\Gamma_L\to\Gamma_{L,A}=\ZZ\ee_z\oplus\bigoplus_{[p]\in P_L/{\sim_A}}\ZZ\ee_{[p]}$ be the surjective homomorphism defined by $\pi_A(\ee_z)=\ee_z$ and $\pi_A(\ee_p)=\ee_{[p]}$. Then $\varphi_A$ maps every variable to a monomial, and the induced homomorphism of exponent groups is $\pi_A\circ\betab$. Hence $\ker\varphi_A=I_{\ker_\ZZ(\pi_A\circ\betab)}$, and it suffices to show that $\ker_\ZZ(\pi_A\circ\betab)=\Lambda_A$.
		
		We claim that $\ker_\ZZ\pi_A$ is generated by the elements $\ee_{p_i}-\ee_{p'_i}$, $i\in A$. These elements belong to $\ker_\ZZ\pi_A$, since $[p_i]=[p'_i]$. Conversely, let $\zz=\lambda\ee_z+\sum_{p\in P_L}\lambda_p\ee_p\in\ker_\ZZ\pi_A$. Then $\lambda=0$ and $\sum_{p\in\mathcal C}\lambda_p=0$ for every equivalence class $\mathcal C$ of $\sim_A$, so $\zz$ is an integral linear combination of elements $\ee_p-\ee_{p''}$ with $p\sim_Ap''$. Finally, if $p\sim_Ap''$, then there exists a sequence $p=q_0,q_1,\ldots,q_l=p''$ such that $\{q_{j-1},q_j\}=\{p_i,p'_i\}$ for some $i\in A$, for every $j\in[l]$, and so $\ee_p-\ee_{p''}=\sum_{j=1}^l(\ee_{q_{j-1}}-\ee_{q_j})$ is an integral linear combination of the elements $\ee_{p_i}-\ee_{p'_i}$, $i\in A$. This proves the claim.
		
		Let $\mathrm P_A$ be the subgroup of $\ZZ^L$ generated by $\rhob_i$, $i\in A$, so that $\Lambda_A=\Lambda_\emptyset+\mathrm P_A$. Since $\betab(\rhob_i)=\ee_{p_i}-\ee_{p'_i}$, the claim says that $\ker_\ZZ\pi_A=\betab(\mathrm P_A)$. Therefore, using Lemma~\ref{lem:kernel-beta},
		\[
		\ker_\ZZ(\pi_A\circ\betab)=\betab^{-1}\bigl(\betab(\mathrm P_A)\bigr)=\mathrm P_A+\ker_\ZZ\betab=\mathrm P_A+\Lambda_\emptyset=\Lambda_A.
		\]
		Here, the second equality holds because $\betab(\zz)\in\betab(\mathrm P_A)$ if and only if $\zz-\mathbf r\in\ker_\ZZ\betab$ for some $\mathbf r\in\mathrm P_A$. Hence $\ker\varphi_A=I_{\Lambda_A}$, and $J_A\subseteq I_{\Lambda_A}$ by Proposition~\ref{prop:saturation-JA}. Finally, $T_A\cong S_L/\ker\varphi_A=S_L/I_{\Lambda_A}$, and Lemma~\ref{lem:integral-basis-LA} gives $\dim T_A=|L|-\rank_\ZZ\Lambda_A=(2m+2)-(m+|A|)=m+2-|A|$.
	\end{proof}
	
	\begin{Corollary}\label{cor:toric}
		Let $L$ be a thin distributive lattice of rank $m+1$, and let $A\subseteq[m]$. Then $\ker\varphi_A=J_A$ if and only if $A$ is admissible. In this case $S_L/J_A\cong T_A$; that is, $S_L/J_A$ is isomorphic to the affine semigroup ring obtained from Hibi's parametrization of $K[L]$ by identifying the variables $y_{p_i}$ and $y_{p'_i}$ for every $i\in A$.
	\end{Corollary}
	
	\begin{proof}
		If $A$ is admissible, then $J_A=I_{\Lambda_A}=\ker\varphi_A$ by Theorem~\ref{thm:admissible} and Proposition~\ref{prop:toric-kernel}. Conversely, if $\ker\varphi_A=J_A$, then $J_A$ is a prime ideal, since $T_A$ is a domain, and so $A$ is admissible by Theorem~\ref{Thm for primality}.
	\end{proof}
	
	\begin{Remark}\label{rem:CBN}
		The parametrization $\varphi$ has a natural probabilistic meaning. Let $K=\mathbb R$ and regard the elements of $P_L$ as events which can only occur in an order compatible with $P_L$, that is, an event $p$ can occur only after all the events $p'<p$ have occurred. This is the \emph{conjunctive Bayesian network} on $P_L$ introduced in \cite{BES}. The possible states are the sets of events which have occurred, namely the poset ideals $\ideal\alpha$, $\alpha\in L$, and the parameters are numbers $\theta_p\in(0,1)$, $p\in P_L$, where $\theta_p$ is the probability that $p$ occurs given that all the events $p'<p$ have occurred. For $\alpha\in L$, let $q_\alpha$ be the probability that all the events of $\ideal\alpha$ have occurred. Then $q_\alpha=\prod_{p\in\ideal\alpha}\theta_p$, see the proof of \cite[Theorem~10]{BES}, that is, $q_\alpha$ is the value of $\varphi(x_\alpha)$ at $z=1$ and $y_p=\theta_p$. Beerenwinkel, Eriksson and Sturmfels observed that, in these coordinates, the model is the toric variety of the Hibi ring of $L$: the relations among the probabilities $q_\alpha$ are generated by $q_{0_L}=1$ and by the join-meet relations $q_\alpha q_\beta=q_{\alpha\wedge\beta}q_{\alpha\vee\beta}$; see \cite[Corollary~11]{BES}. This is the case $A=\emptyset$ of the discussion below.
		
		Our results describe what happens when some conditional probabilities are required to coincide. Since $\ideal{c_i}=\ideal{a_i}\sqcup\{p_i\}$ and $\ideal{d_i}=\ideal{a_i}\sqcup\{p'_i\}$, we have $q_{c_i}-q_{d_i}=\bigl(\prod_{p\in\ideal{a_i}}\theta_p\bigr)(\theta_{p_i}-\theta_{p'_i})$, and so $q_{c_i}=q_{d_i}$ if and only if $\theta_{p_i}=\theta_{p'_i}$, since the parameters are positive. Thus the rank identification $x_{c_i}=x_{d_i}$ corresponds to requiring that the two events $p_i$ and $p'_i$, which distinguish the two states of rank $i$, have the same conditional probability. Fix $A\subseteq[m]$ and impose these conditions for every $i\in A$. Then $\theta_p$ only depends on the class $[p]$ of $p$ with respect to $\sim_A$, and $q_\alpha$ is the value of $\varphi_A(x_\alpha)$ at $z=1$ and $y_{[p]}=\theta_{[p]}$. If $f\in S_L$ is homogeneous of degree $r$, then $\varphi_A(f)=z^r\cdot\varphi_A(f)|_{z=1}$; since a nonzero polynomial cannot vanish on a product of open intervals, it follows that $f$ vanishes at all the vectors $(q_\alpha)_{\alpha\in L}$ of this submodel if and only if $f\in\ker\varphi_A$. Hence Corollary~\ref{cor:toric} can be read as follows: the homogeneous relations among the probabilities of the submodel are generated by the obvious ones, namely the join-meet relations and the linear relations $q_{c_i}=q_{d_i}$, $i\in A$, if and only if $A$ is admissible. In this case, the ring $T_A\cong S_L/J_A$ is Cohen--Macaulay of dimension $m+2-|A|$, with the same $h$-polynomial as the Hibi ring of $L$, by Remark~\ref{rem:homological}. For an arbitrary $A$, the homogeneous relations of the submodel are given by the saturation $J_A:X_L^\infty$, by Propositions~\ref{prop:saturation-JA} and~\ref{prop:toric-kernel}.
		
	\end{Remark}
	
	\begin{Example}\label{ex:CBN}
		Consider four events $1,2,3,4$ such that the event $2$ can occur only after the event $1$, and the event $4$ only after the events $1$ and $3$. Thus $P$ is the poset on $\{1,2,3,4\}$ with $1<2$, $1<4$ and $3<4$, and $L=\cJ(P)$ consists of the poset ideals
		\[
		\emptyset;\quad\{1\},\ \{3\};\quad\{1,2\},\ \{1,3\};\quad\{1,2,3\},\ \{1,3,4\};\quad\{1,2,3,4\},
		\]
		listed by rank. Hence $L$ is a thin distributive lattice of rank $4$, and $\cT(L)=\{2\}$, since $\{1\}\vee\{3\}=\{1,3\}=\{1,2,3\}\wedge\{1,3,4\}$; see Figure~\ref{fig:CBN}.
		
		\begin{figure}[htbp]
			\centering
			\begin{tikzpicture}[x=1.5cm,y=1.1cm,dot/.style={circle,fill,inner sep=1.3pt},every node/.style={font=\small}]
				\fill[gray!20] (0,1) -- (1,2) -- (0,3) -- (-1,2) -- cycle;
				\coordinate (e0) at (-1,0);
				\coordinate (e3) at (-2,1); \coordinate (e1) at (0,1);
				\coordinate (e13) at (-1,2); \coordinate (e12) at (1,2);
				\coordinate (e134) at (-2,3); \coordinate (e123) at (0,3);
				\coordinate (e1234) at (-1,4);
				\foreach \s/\t in {e0/e3,e0/e1,e3/e13,e1/e13,e1/e12,e13/e134,e13/e123,e12/e123,e134/e1234,e123/e1234} \draw (\s)--(\t);
				\foreach \s in {e0,e3,e1,e13,e12,e134,e123,e1234} \node[dot] at (\s) {};
				\node[below] at (e0) {$\emptyset$};
				\node[left=2pt] at (e3) {$\{3\}$};
				\node[right=2pt] at (e1) {$\{1\}$};
				\node[left=2pt] at (e13) {$\{1,3\}$};
				\node[right=2pt] at (e12) {$\{1,2\}$};
				\node[left=2pt] at (e134) {$\{1,3,4\}$};
				\node[right=2pt] at (e123) {$\{1,2,3\}$};
				\node[above] at (e1234) {$\{1,2,3,4\}$};
			\end{tikzpicture}
			\caption{The thin distributive lattice $L=\cJ(P)$ of Example~\ref{ex:CBN}. The elementary diamond of the turning rank $2$ is shaded.}
			\label{fig:CBN}
		\end{figure}
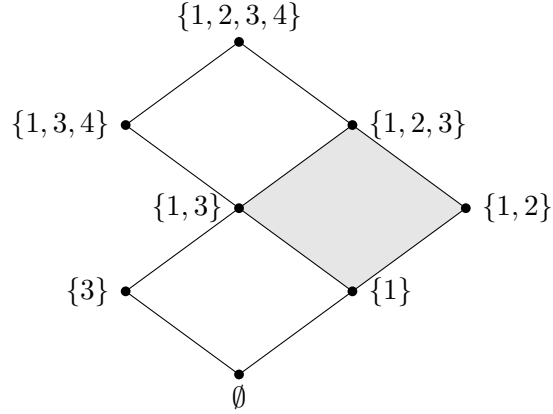
		
		We write $x_\emptyset,x_1,x_3,x_{12},\ldots$ for the corresponding variables, and similarly $q_U=\prod_{p\in U}\theta_p$ for the probabilities; for instance, $q_{134}=\theta_1\theta_3\theta_4$. The two elements of rank $1$, $2$ and $3$ differ by the events $1$ and $3$, $2$ and $3$, and $2$ and $4$, respectively, so the rank identifications of ranks $1$, $2$ and $3$ correspond to the conditions $\theta_1=\theta_3$, $\theta_2=\theta_3$ and $\theta_2=\theta_4$. Since $\cT(L)=\{2\}$, the admissible subsets of $[3]$ are $\emptyset$, $\{1\}$ and $\{3\}$. For instance, if $A=\{1\}$, then every homogeneous relation among the probabilities follows from the join-meet relations and from $q_1=q_3$; for example, $q_\emptyset q_{13}=q_1^2$ follows from $q_1q_3=q_\emptyset q_{13}$.
		
		Now let $A=\{2\}$. If $\theta_2=\theta_3$, then $q_\emptyset q_{1234}=\theta_1\theta_3^2\theta_4=q_3q_{134}$, but the binomial $x_\emptyset x_{1234}-x_3x_{134}$ does not belong to $J_A$, by the proof of Proposition~\ref{prop:turning-relation-nonprime}. Similarly, let $A=\{1,3\}$. If $\theta_1=\theta_3$ and $\theta_2=\theta_4$, then $q_\emptyset q_{1234}=\theta_1^2\theta_2^2=q_{12}^2$, but $x_\emptyset x_{1234}-x_{12}^2\notin J_A$, by the proof of Proposition~\ref{prop:neighboring-relations-nonprime}. In both cases, the join-meet relations and the linear relations $q_{c_i}=q_{d_i}$, $i\in A$, do not generate all the homogeneous relations of the submodel.
	\end{Example}

\end{document}